\documentclass[11pt,usenames,dvipsnames]{article}
\usepackage[T1]{fontenc}
\usepackage{lmodern}
\usepackage{dsfont}
\usepackage{multicol}
\usepackage{enumerate}
\usepackage{authblk}
\usepackage{amssymb,amsmath,amsthm,epsfig,textcomp,amsthm}
\usepackage[colorlinks=true,breaklinks=true,linkcolor=MidnightBlue!100,citecolor=Maroon!100]{hyperref}
\usepackage{graphics,psfrag,graphicx,color}
\usepackage[graphicx]{realboxes}
\numberwithin{equation}{section}
\usepackage[normalem]{ulem}

\usepackage{enumitem}  % itemize without bullets
\setlist{  
  listparindent=\parindent,
  parsep=0pt,
}
\usepackage[x11names,table]{xcolor}
\usepackage{booktabs}
\usepackage{rotating}
\usepackage{varwidth}
\usepackage{comment}
\usepackage{float}
\floatstyle{boxed} 

\usepackage{tabularx,ragged2e,booktabs,caption}
\newcolumntype{C}[1]{>{\Centering}m{#1}}

\usepackage{multirow}
\usepackage{algorithm}
\usepackage{algpseudocode}
\usepackage{etoolbox}
\makeatletter
\newcommand*{\algrule}[1][\algorithmicindent]{%
  \makebox[#1][l]{%
    \hspace*{.2em}% <------------- This is where the rule starts from
    \vrule height .75\baselineskip depth .25\baselineskip
  }
}

\newcount\ALG@printindent@tempcnta
\def\ALG@printindent{%
    \ifnum \theALG@nested>0% is there anything to print
    \ifx\ALG@text\ALG@x@notext% is this an end group without any text?
    \else
    \unskip
    \ALG@printindent@tempcnta=1
    \loop
    \algrule[\csname ALG@ind@\the\ALG@printindent@tempcnta\endcsname]%
    \advance \ALG@printindent@tempcnta 1
    \ifnum \ALG@printindent@tempcnta<\numexpr\theALG@nested+1\relax
    \repeat
    \fi
    \fi
}
\patchcmd{\ALG@doentity}{\noindent\hskip\ALG@tlm}{\ALG@printindent}{}{\errmessage{failed to patch}}
\patchcmd{\ALG@doentity}{\item[]\nointerlineskip}{}{}{} % no spurious vertical space
\usepackage{mathrsfs}
\usepackage{mathtools}

\newcommand{\pdual}[1]{\left\langle#1\right\rangle}

\newcommand{\jump}[1]{ [ \! [ {#1} ] \! ] }

\newcommand{\what}[1]{\widehat{#1}}
\newcommand{\kap}{\kappa}
\newcommand{\bq}{\boldsymbol{q}}

\newcommand{\bsigma}{\boldsymbol{\sigma}}
\newcommand{\g}{\overline{g}}
\newcommand{\compD}{\Omega_h}
\newcommand{\erroru}{\varepsilon^u}
\newcommand{\errorsigma}{\boldsymbol{\varepsilon}^{\boldsymbol\sigma}}
\newcommand{\errorq}{\boldsymbol{\varepsilon}^{\boldsymbol{q}}}
\newcommand{\errortu}{\varepsilon^{\widehat{u}}}
\newcommand{\errortq}{\boldsymbol{\varepsilon}^{\widehat{\boldsymbol{q}}}}
\newcommand{\projerrorsigma}{\Lambda_{\boldsymbol{\sigma}}}
\newcommand{\projerrorq}{\Lambda_{\boldsymbol{q}}}
\newcommand{\projerroru}{\Lambda_{u}}

\newcommand{\triple}[1]{|\!|\!|{#1}|\!|\!|}

\newcommand{\mc}[1]{\mathcal{#1}}
\newcommand{\mr}[1]{\mathrm{#1}}
\newcommand{\md}[1]{\mathds{#1}}
\newcommand{\mbf}[1]{\boldsymbol{#1}}
\newcommand{\bsy}[1]{\boldsymbol{#1}}
\newcommand{\cO}{\mathcal{O}}

\newcommand\bH{\boldsymbol{H}}
\newtheorem{thm}{Theorem}
\newtheorem{lem}{Lemma}
\newtheorem{crl}{Corollary}

\title{Error analysis of an unfitted HDG method \\ for a class of non-linear elliptic problems.}

\author[1,2,4]{Nestor S\'anchez }
\author[3]{Tonatiuh S\'anchez-Vizuet}
\author[1,2]{Manuel Solano}

\affil[1]{{\small Departamento de Ingenier\'ia Matem\'atica, Universidad de Concepci\'on, Concepci\'on, Chile.}}
\affil[2]{{\small Centro de Investigaci\'on 
en Ingenier\'ia Matem\'atica (CI$^2$MA),Universidad de Concepci\'on, Concepci\'on, Chile.}}
\affil[3]{{\small Department of Mathematics, The University of Arizona, USA.}}
\affil[4]{{\small Instituto de Matem\'aticas, Unidad Juriquilla. Universidad Nacional Aut\'onoma de M\'exico.}}

\begin{document}

%%%%%%%%%%%%%%%%%%%%%%%%%%%%%%%%%%%%%%%%%%%%%%%%%%%%%%%%%
\maketitle
%%%%%%%%%%%%%%%%%%%%% Abstract %%%%%%%%%%%%%%%%%%%%%%%%%
\begin{abstract}
We study Hibridizable Discontinuous Galerkin (HDG) discretizations for a class of non-linear interior elliptic boundary value problems posed in curved domains where both the source term and the diffusion coefficient are non-linear. We consider the cases where the non-linear diffusion coefficient depends on the solution and on the gradient of the solution. To sidestep the need for curved elements, the discrete solution is computed on a polygonal subdomain that is not assumed to interpolate the true boundary, giving rise to an unfitted computational mesh. We show that, under mild assumptions on the source term and the computational domain, the discrete systems are well posed. Furthermore, we provide \textit{a priori} error estimates showing that the discrete solution will have optimal order of convergence as long as the distance between the curved boundary and the computational boundary remains of the same order of magnitude as the mesh parameter.
\end{abstract}



%%%%%%%%%%%%%%%%%%%%%%%%%%%%%%%%%%%%%%%%%%%%%%%%%%%%%%%%%%

% ======================
\section{Introduction}
% ======================
In this work we will study a discretization based on the hybridizable discontinous Galerkin (HDG) method \cite{CoGoLa2009} for a class of quasilinear elliptic boundary value problems of the form
\begin{subequations}\label{eq:BVP}
\begin{alignat}{6}
-\nabla\cdot\left(\kappa\,\nabla u\right) =\,& f(u) & \qquad & \text{ in } \Omega \\
u =\,& g & \qquad & \text{ on } \Gamma := \partial\Omega,
\end{alignat}
where the domain $\Omega \subset \mathbb{R}^d$ ($d=2,3$) is not necessarily polygonal/polyhedral, the diffusion coefficient, $\kappa$, is a positive function that depends on the solution, $u$, in one of the following functional forms
\begin{equation}\label{eq:BVPkappa}
\kappa = \left\{ \begin{array}{c} \kappa(u) \\ \kappa(\nabla u)\end{array}\right.,
\end{equation}
\end{subequations}
and that will be assumed to be a bounded and Lipschitz---in a sense that will be made precise in due time. In addition, the source  function $f$ will be taken to be a Lipschitz-continuous mapping from $L^2(\Omega)$ to $L^2(\Omega)$, so that there exists $L_{f}> 0$ such that
    \begin{equation}\label{Lipschitz_F}
    \|f(u_1) - f(u_2)\|_{\Omega} \leq L_{f} \, \|u_1-u_2\|_{\Omega} \qquad \forall \, u_1, u_2 \in L^2(\Omega).
    \end{equation}
The authors became interested boundary value problems of this form through the study of magnetic equilibrium configurations in cylindrically symmetric fusion reactors. In the context of plasma equilibrium, if the location of the plasma within the reactor is assumed to be known, the equilibrium condition between magnetic and hydrostatic forces results in an equation like the one above, known as the Grad-Shafranov (or Grad-Shafranov-Schl\"uter) equation \cite{GrRu:1958, LuSc:1957,Shafranov1958}. The domain $\Omega$ in which the equation holds, corresponds to the region of the reactor occupied by the plasma. This region in general has a non-polygonal, piecewise smooth Lipschitz boundary $\Gamma := \partial\Omega$ with a small number of corners---known in the plasma literature as \textit{x-points}.

In this model, the unknown $u$ is the stream function of the poloidal magnetic field, the source term $f$ is a nonlinear function of $u$ accounting for the effects of the hydrostatic pressure and total electric currents present in the device, and the coefficient $\kappa$ encodes the magnetic properties of the system. The case where $\kappa$ is a constant leads to a semi-linear equation for which an HDG discretization was proposed and implemented in  \cite{SaSo2018, SaCeSo2019}, and analyzed in detail in \cite{SaSaSo2019}. However, in the presence of ferroelectric materials, the permeability is affected by the total magnetic field $\mathbf B$---which is proportional to the gradient of $u$---and the coefficient then takes the form $\kappa = \kappa(\nabla u)$, leading to a quasi-linear equation that requires the more detailed treatment that will be the subject of this article. Some theoretical studies of the HDG method applied to quasilinear problems have been pursued recently \cite{DaFe:2014,GaSe:2015,GS2}, however these efforts are limited to polygonal domains. Moreover, the first reference does not consider non-linearities of the form $\kappa(\nabla u)$, while in \cite{GaSe:2015,GS2} the authors analyzed an augmented HDG discretization for a strictly quasi-linear problem arising from a non-linear Stokes flow using an approach based on a nonlinear version of the Babuska-Brezzi theory. \textcolor{black}{A remarkable effort involving HDG and interpolatory methods for semilinear problems was recently carried out in \cite{ChCoSiZh2019,CoSiZh2019}, as well as some recent contribution for the steady state incompressible  Navier--Stokes equations \cite{Leng2021}. An interesting application of HDG to some fully non-linear time dependent equations related to shallow water models was carried out in \cite{Marche2020}, while an HDG scheme for the equations of magnetohydrodynamics was analyzed in \cite{QiSh2019}. However in these works, either the domain is considered to be polygonal or non-linearities are restricted to the source term---or both.} As we will show, our analysis will be valid for both quasi-linear and semi-linear problems, will not require an augmented formulation and the domain may be piecewise smooth. 

Due to the non-polygonal nature of the domain of definition, any discretization scheme employed must handle the additional geometric complexity of the problem. The HDG scheme that we propose deals with the geometry using a method introduced within the context of HDG discretizations for linear elliptic equations in \cite{CoSo2012} and given firm theoretical justification in \cite{CoQiuSo2014}. The idea  consists of posing the discrete problem in a polygonal subdomain $\Omega_h\subset\Omega$ and transferring the boundary data to the computational boundary $\Gamma_h := \partial\Omega_h$ through a mechanism that involves a line integral of the numerical flux. This process requires reformulating the equation in the form of a first order system (mixed form), but greatly simplifies the computational implementation by allowing the use of a shape regular triangulation and avoiding the need for high order curved elements on the boundary. Moreover, in applications where the flux is a physically relevant variable---as is the case of the plasma problem---the direct discretization of the flux required by the mixed formulation is an additional advantage of the HDG formulation. 

Depending on the location of the non-linearity within the equation, the analysis required to generalize \cite{CoQiuSo2014,CoSo2012} into non-linear problems of the form \eqref{eq:BVP} can be split into two independent parts. The case where $\kappa$ is independent of the solution and only the source $f$ is a function of $u$, has already been dealt with in \cite{SaSaSo2019}. The current communication will deal with situations where the source term is independent of $u$ and $\kappa$ takes one of the forms \eqref{eq:BVPkappa}, each of which must also be analyzed separately. We will start by introducing basic notation in Section \ref{sec:preliminaries}, where and the idea of the extended domains and transfer paths will be presented. In this section we will also  describe some geometric hypotheses on the the computational domain that will be useful for the error analysis. Having established the basic setting, we then proceed to study separately the HDG discretizations for the case when the diffusion coefficient is a function of $u$ only (Section \ref{sec:DependenceOnU}), and the case where the diffusion coefficient depends on $\nabla u$ (Section \ref{sec:DependenceOnGradU}). In these sections the well posedness of the corresponding discrete HDG formulations are established, and \textit{a priori} error analyses on the discretizations are performed.

% ===============================================
\section{Preliminaries}\label{sec:preliminaries}
% ===============================================
%
% ====================================================================
\subsection{Computational domains and admissible triangulations}\label{sec:ComputationalDomain}
% ===================================================================
%
Given the domain $\Omega$ where \eqref{eq:BVP} is posed, we will  define a family of polygonal subdomains and admissible triangulations approximating $\Omega$ where we will ultimately pose our discretization. First, consider a family of simply connected domains $\{\Omega_\alpha\}_{\alpha>0}$  such that, for every $\alpha$ the following conditions hold: (1) $\Omega_\alpha \subseteq \Omega$, (2) the boundary $\Gamma_\alpha :=\partial \Omega_\alpha$ is a polygon, and (3) for every $\epsilon>0$ there are infinitely many indices $\alpha$ such that $\lambda(\Omega\setminus\Omega_\alpha)<\epsilon$. In the preceding expression $\lambda(\cdot)$ denotes the Lebesgue measure. These conditions ensure that the family of subdomains $\{\Omega_\alpha\}_{\alpha>0}$ will exhaust $\Omega$.

Having built the family $\{\Omega_\alpha\}_{\alpha>0}$ satisfying all the conditions above, the next step is to define a family of admissible simplicial triangulations $\{\mc{T}_{h}\}_{h>0}$. To be considered admissible, a triangulation $\mc{T}_{h}$ must be such that: (1) $\mc{T}_{h}$ is a triangulation for at least one $\Omega_h \in \{\Omega_\alpha\}_{\alpha>0}$ (we will identify both $\mc{T}_h$ and the respective domain $\Omega_h$ with the same subscript $h$, adding copies of $\Omega_h$ to $\{\Omega_\alpha\}_{\alpha>0}$ if necessary to account for different triangulations of the same domain)  (2) it is shape regular, meaning that there exists $\beta>0$ such that for all elements $T\in \mc{T}_h$ and all $h>0$, $h_T / \rho_T \leq \beta $, where $h_T$ is the diameter of  $T$ and $\rho_T$ is the diameter of the largest ball contained in $T$, and (3) for every $T\in \mc{T}_h$ such that $T\cap\Gamma_h \neq \varnothing$, the maximum distance between $\boldsymbol x \in T\cap\Gamma_h$ and $\boldsymbol y \in \Gamma$ is of the same order of magnitude as the element diameter $h_T$. More precisely, if $d_{loc}:= \max\{d(\boldsymbol x,\boldsymbol y): \boldsymbol x \in T\cap\Gamma_h \text{ and } \boldsymbol y \in \Gamma\}$ then $d_{loc}= \mathcal O (h_T)$. This last requirement, which will be referred to as the \textit{local proximity condition} and is depicted schematically in Figure \ref{fig:LocalProximity}, is of key importance for the transfer process that will be defined later on. 
\begin{figure}
\centering
\begin{tabular}{ccc}
\includegraphics[width = 0.25\linewidth]{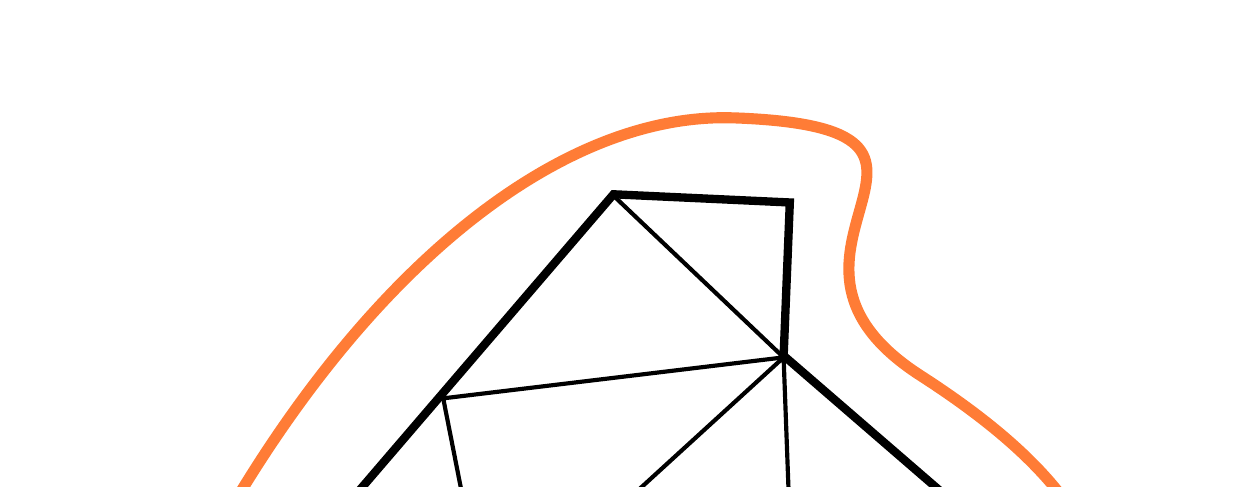} \qquad  &
\qquad 
\includegraphics[width = 0.25\linewidth]{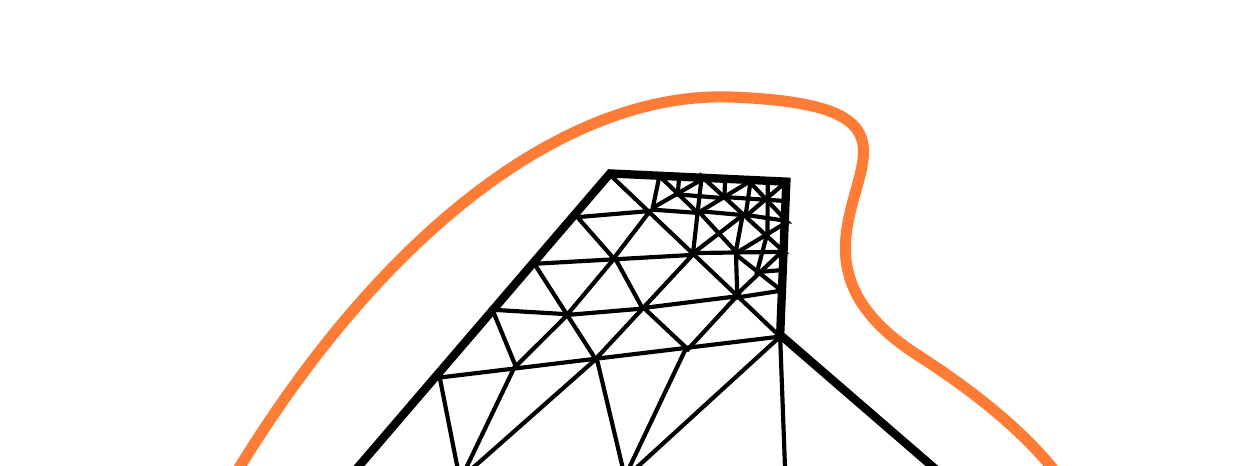} \qquad  &
\qquad 
\includegraphics[width = 0.25\linewidth]{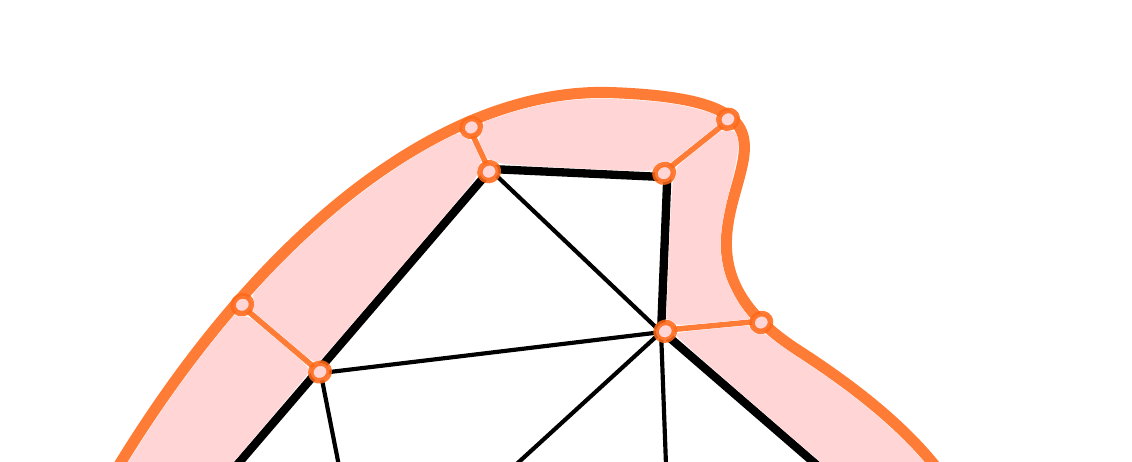} 
\end{tabular}
\caption{ Left and center: The proximity condition ensures that the distance between the computational and the physical boundaries remains always of the same order of magnitude as the local element diameter. The schematic shows close ups to the boundary of two triangulations of the same domain: an admissible triangulation satisfying the local proximity condition (left), and inadmissible one that violates the local proximity condition (center). Right: An admissible triangulation and one possible arrangement of extension patches $T^e_{ext}$ (shaded in the figure) defined on the region $\Omega\setminus\Omega_h$.}
\label{fig:LocalProximity}
\end{figure} 

For every element $T$ in a particular triangulation, we will denote by $\boldsymbol{n}_T$ the outward unit normal vector to $T$, or simply $\boldsymbol{n}$ instead of $\boldsymbol{n}_T$ whenever the context prevents any confusion. As it is conventional, we will denote the mesh parameter as $\displaystyle h:=\max_{T\in \mc{T}_h} h_T$, which will be assumed to be smaller than one for the sake of simplicity. We will denote by $e$ any face of a simplex and its length by $h_e$. Moreover, we will talk about an \textit{interior face} $e$ if there are two elements $T^+$ and $T^-$ in the triangulation $\mc{T}_h$ such that $e = \partial T^+ \cap \partial T^-$. The set of all interior faces will be denoted by $\mc{E}_h^\circ$. In a similar manner, we will talk about a \textit{boundary face} $e$ if there is an element $T\in \mc{T}_h$ such that $e=\partial T \cap \Gamma_h$; the set of boundary faces will be denoted by $ \mc{E}_h^{\partial}$. Note that with these definitions, the entirety of the faces of the triangulation denoted by $\mc{E}_h$ (often referred to as the \textit{skeleton} of the mesh) can then be decomposed as $\mc{E}_h= \mc{E}_h^\circ \cup \mc{E}_h^{\partial}$.

Throughout the paper, we will be working with functions that in general will not be continuous across mesh elements. For a scalar-valued function we will use the symbol $\jump{w}:=
w^+-w^-$ to refer to its jump across any given interior face. At the boundary faces, the jump will be defined as $\jump{w}:=w-\varphi_h$, where $\varphi_h$ is the approximation of the boundary data at $\Gamma_h$ that will be defined later. In the case of vector-valued functions $\mbf{v}$, we will be interested in the discontinuity of its normal component across interior faces, which will be denoted by $\jump{\mbf{v}}:=\mbf{v}^+\cdot \mbf{n}^+ + \mbf{v}^-\cdot \mbf{n}^-$.

\textbf{A remark on the local proximity condition and mesh refinement:} The local proximity condition limits the minimum size that the elements near the boundary of an admissible triangulation can attain. Therefore, mesh refinement in this context must be understood as moving through a sequence of computational domains in the set $\{\Omega_h\}_{h>0}$ and their corresponding admissible triangulations in $\{\mathcal{T}_h\}_{h>0}$ as the parameter $h\to 0$. As it will be shown later, the error estimates will not depend on the particular domain $\Omega_h$ or triangulation $\mathcal{T}_h$ as long as the three requirements on the mesh stated above are satisfied. Possible ways of building sequences of admissible triangulations and computational domians have been detailed in \cite{CoSo2012} for uniform meshes and in \cite{SaCeSo2019} for adaptively refined triangulations.
%
% =========================================================
\subsection{The extended domain \label{sec:ExtendedDomain}}
% =========================================================
Having defined the family of polygonal subdomains and admissible triangulations on which the discretization will be performend, we will now proceed to detail the process through which the boundary information will be transferred from the boundary into the computational domain. In order to do that we will have to tessellate the region enclosed between the two boundaries $\Gamma$ and $\Gamma_h$ as follows. 

Given a triangulation $\mc{T}_h$ of the computational domain  $\Omega_h$ and a boundary face $e \in \mc{E}_h^{\partial}$, we will denote by $T^e$ the unique element of $\mc{T}_h$ such that $e \cap \overline{T^e} = e$. To every point $\mbf{x} \in e$, we will associate---in a smooth fashion---a point $\overline{\mbf{x}} \in \Gamma$ and set $l(\mbf{x})=|\mbf{x}-\overline{\mbf{x}}|$. We will define the \textit{extension patch} $T^e_{ext}$ as
	\begin{equation*}
	T^e_{ext} := \{ \mbf{x} +s\mbf{t} : 0 \leq s \leq l(\mbf{x}) , \mbf{x} \in e \},
	\end{equation*}
where $\mbf{t}=\mbf{t}(\mbf{x})$ is the unit vector anchored at $\boldsymbol x$ and pointing in the direction of $\overline{\mbf{x}}$.  With this notation, the line segment connecting $\mbf{x}$ to $\overline{\mbf{x}}$ can be parameterized by
    \begin{equation*}
	\sigma_{\mbf{t}}(\mbf{x}) := \{ \mbf{x} +s\mbf{t} : s\in [0, l(\mbf{x})]\}.
	\end{equation*}
%
%The point $\overline{\mbf{x}}\in \Gamma$ can be specified in several ways. For instance, it can be a point that minimizes the distance between $\mbf{x}$ and $\Gamma_h$. However, in that case $\overline{\mbf{x}}$ might be not unique and also the union of all such extension patches $T^e_{ext}$ may not cover the set $\compD^c := \Omega \setminus \overline{\compD}$ entirely if $\Omega$ is not convex. A second possibility is to set $\overline{\mbf{x}}$ to be the closest intersection between $\Gamma$ and the ray starting at $\mbf{x}$ and pointing in the direction of the vector $\mbf{n}$ normal to the face $e$ where $\mbf{x}$ belongs. In that case, $\mbf{t}=\mbf{n}$ and $T^e_{ext}$ may not cover $\compD^c$. \textcolor{black}{Moreover, $l(\mbf{x})$ could be extremely large compared to the mesh size [\textbf{NOTE}: This would not happen when the local proximity condition is enforced]}.%
The point $\overline{\mbf{x}}\in \Gamma$ and therefore the vector $\boldsymbol{t}(\boldsymbol x)$ can be specified in several ways. Here, we will consider that the point has been determined in such a way that
\begin{equation}\label{eq:S1}
\mbf{t}(\mbf{x}) = \mbf{n} \text{ for all }\mbf{x}\in e.
\end{equation}
This assumption is made with the sole purpose of making the analysis simpler, it can in fact be relaxed to the existence of a constant $a_0$ such that $0<a_0\leq \boldsymbol t(\boldsymbol x)\cdot\boldsymbol n$ for every $\boldsymbol x$ belonging to a boundary edge. The numerical method described here is remarkably robust with respect to the method used to choose $\boldsymbol t(\boldsymbol x)$. Previously, the direction had been determined using the algorithm proposed by \cite{CoSo2012}, which assigns $\overline{\boldsymbol{x}}$ in such a way that the three following conditions are satisfied: (1) $\overline{\boldsymbol{x}}$ is unique, (2) any two different line segments $\sigma_{\mbf{t}}$ do not intersect each other inside $T^e_{ext}$, and (3) the segments $\sigma_{\mbf{t}}$ do not intersect the interior of $\Omega_h$. As it is proven in the aforementioned reference, these three conditions guarantee that the union of $T^e_{ext}$ completely covers $\compD^c := \Omega \setminus \overline{\compD}$. An alternate method was used and tested numerically in \cite{SaSo2018,SaCeSo2019}, where $\boldsymbol t(\boldsymbol x)$ was determined using a weighted average of the normal vectors from neighboring  boundary edges.   

For an extended patch $T^e_{ext}$ and a mesh element $T^e\in\mathcal T_h$ sharing a boundary face $e\in\mathcal E^\partial_h$, we denote by $h_e^{\perp}$ (resp. $H_e^{\perp}$) the largest distance between a point inside $T_e$ (resp. $T_{ext}^e$) and the plane determined by the face $e$. The ratio between these two distances will be denoted by $r_e := H_e^{\perp} / h_e^{\perp}$ and the maximum such ratio taken over all the boundary edges will be denoted by $\displaystyle R_{\mathcal T_h}:= \max_{e\in \mc{E}_h^{\partial}} r_e$. Note that the local proximity condition will ensure the existence of a constant R, that we will call the \textit{proximity constant}, independent of the particular triangulation $\mathcal T_h$, such that
\begin{equation}\label{eq:S2}
 R: = \sup_{\mathcal T_h} R_{\mathcal T_h}  = \sup_{\mathcal T_h}\left\{ \max_{e\in \mc{E}_h^{\partial}}\, r_e\right\}< \infty.  
\end{equation}
We will also define the class of non-trivial vector-valued polynomials of degree at most $k$ defined in and across both patches as
    \begin{equation*}
\mathcal V^k := \left\{ \boldsymbol p \in \mathbb [\mathbb{P}_k(T^e_{ext} \cup T^e)]^2\, : \,  \boldsymbol p\cdot \boldsymbol{n}_e\neq\boldsymbol 0 \right\}.
    \end{equation*}
We can then introduce, for all those elements with a non-empty intersection with the computational boundary, the element-wise constants
    \begin{equation*}
	C^e_{ext} := \dfrac{1}{\sqrt{r_e}} \sup_{\bsy{\chi}\in \mathcal V^k} \dfrac{\|\bsy{\chi}\cdot\boldsymbol n_e \|_{T_{ext}^e} }{\|\bsy{\chi}\cdot\boldsymbol n_e\|_{T^e}} \quad \text{ and } \quad
	C^e_{inv} := h_e^{\perp} \sup_{\bsy{\chi}\in \mathcal V^k} \dfrac{\|\nabla \bsy{\chi}\cdot \boldsymbol n_e  \|_{T^e} }{\|\bsy{\chi}\cdot\boldsymbol n_e \|_{T^e}}, 
	\end{equation*}
where, in abuse of notation, $\boldsymbol n_e$ is a constant vector field defined in $T^e_{ext} \cup T^e\cup\, e$  that coincides with the unit exterior normal vector associated to the face $e$ and pointing in the direction of the extension patch. Above, the norms $\|\cdot\|_{T^e_{ext}}$ and $\|\cdot\|_{T^e}$ are the standard $L^2$ norms supported on the extension patch $T^e_{ext}$ and its  neighboring element $T^e$ respectively. In \cite{CoQiuSo2014}, these constants were bounded in terms of the polynomial degree of the approximation, $k$, and the regularity constant of the mesh, $\beta$, as
\begin{equation}\label{eq:Cconstants}
C^e_{ext}\leq C_1(k+1)^2(3\beta + 2)^k\,,\qquad C^e_{inv}\leq C_2k^2,
\end{equation}
where $C_1$ and $C_2$ depend only on the mesh regularity. As we will see below, these constants will determine the magnitude of the proximity constant and therefore the maximum admissible gap between the computational and physical boundaries.

We will also need to make two technical assumptions relating the proximity constant to the and the diffusion coefficient $\kappa$ and the degree of the polynomial approximation. For each $e\in \mc{E}_h^{\partial}$ we will require the following to hold:
\begin{subequations}\label{eq:Assumptions}
	\begin{align}
    H_e^{\perp}\, \leq \frac{1}{3}\, \underline{\kappa} \,\overline{\tau}\,^{-1} ,\label{eq:S3} \\
	\overline{\kap}\, \underline{\kap}^{-1}\, r_e^3\, (C_{ext}^e\,  C_{inv}^e)^2  \leq 1, \label{eq:S4}
	\end{align}
\end{subequations}
where $\underline{\kap}$ and $\overline{\kap}$ are the lower and upper bounds of $\kap$, resp., specified in \eqref{eq:assumptions_kappa}, whereas $\overline{\tau}$ is the maximum of the stabilization parameter $\tau$ of the HDG scheme.
The first of these two conditions, \eqref{eq:S3}, states the well known fact that for small values of the diffusivity, small scale behavior can be expected near the physical boundary, and therefore fine extension patches are required. However, it also provides the additional insight that the distance between the boundaries can be increased at the cost of accepting smaller values of the stabilization factor $\tau$---and hence larger discontinuities in the discrete solution. 
In a similar vein, \eqref{eq:S4} relates the range of values of the diffusion coefficient with the maximum separation between the computational and physical boundaries, thus making sure that the external patches are fine enough to resolve possible boundary behavior induced by large variations in diffusivity over the domain. Moreover, it sets a hard upper limit to the mechanism that allows for a larger separation by decreasing $\tau$.

By combining \eqref{eq:Cconstants} with \eqref{eq:Assumptions} it is not hard to show that, for $k>0$, $H^{\perp}_e$ must be bounded as
\[
H^{\perp}_e \leq \min \left\{ \frac{h^{\perp}_e}{(C_1C_2)^2}\left( \frac{ \overline{\kappa}^{-1}\underline{\kappa} } {k^4(k+1)^4(3\beta+1)^{2k}}\right)^{1/3},\, \frac{1}{3} \underline{\kappa}\,\overline{\tau}^{-1}\right\}.
\]
This expression provides insight into the way in which the physics of the problem---through the range of values for $\kappa$---interacts with the discretization---through the parameters $H^{\perp}_e$, $h^{\perp}_e$, $k$, $\beta$, and $\tau$---and determines the maximum separation between the physical boundary and that of an admissible triangulation. Of particular note is the role played by the polynomial degree of the approximation: for larger values of $k$ the distance between the mesh and the boundary must decrease. The reason for this will become apparent soon, as we will resort to extrapolation to approximate some quantities over the extension patches.
%
% =========================================================
\subsection{{An equivalent mixed formulation and the transfer paths}\label{sec:TransferPaths}}
% =========================================================

Having established the requirements for an admissible triangulation, and defined the extension patches $T^e_{ext}$ in such a way that for each of them there corresponds a single $T^e\in \mathcal T_h$, we can define a way to extend polynomial functions from the computational domain into $\Omega_h^c$. This extension process will enable us to transfer the boundary condition from $\Gamma$ into the computational boundary $\Gamma_h$. Let $p:T^e\to \mathbb R $ be a polynomial function and $T^e_{ext}$ an extension patch associated to $T^e$. We will define the extension $\boldsymbol E_h(p)$ of $p$ to $T^e_{ext}$ by extrapolation as follows:
 \begin{alignat*}{6}
E_h:\quad \mathbb{P}_k(T^e) &\,\longrightarrow &&\, \mathbb{P}_k(T^e\cup T^e_{ext}) \\
p(\boldsymbol y) \; \forall \, \boldsymbol y \in T^e &\, \longmapsto&&\, p(\boldsymbol y) \; \forall \, \boldsymbol y \in T^e\cup T^e_{ext}.
 \end{alignat*}
Where, to keep notation simple, a polynomial function $p$ should be understood as its extrapolation $E_h(p)$  whenever an evaluation outside of $\Omega_h$ is required, which should be clear from the context. For vector-valued polynomial functions, the extension is defined similarly component by component.

To avoid computing on a domain with curved boundary, we wish to pose the boundary vaue problem \eqref{eq:BVP} polygonal subdomain $\compD\subset\Omega$. This simpler geometry can then be discretized by a uniform triangulation of size satisfying the conditions outlined in sections \ref{sec:ComputationalDomain} and \ref{sec:ExtendedDomain}. 

Recasting \eqref{eq:BVP} in mixed form and restricting the resulting equivalent first order system to $\Omega_h$ leads to
	\begin{subequations}\label{eq. Grad-Shaf. numerical domain}
	\begin{align}
	&& && \bq +\kappa\,  \nabla u &= 0 & &\text{ in } \compD \label{eq. G-S. numerical domain_a}\\
	&& && \nabla \cdot \bq  &= f(u) & &\text{ in }   \compD, \label{eq. G-S. numerical domain_b} \\
	&& && u &= \varphi & &\text{ on } \Gamma_h:= \partial  \compD, \label{eq. G-S. numerical domain_c}
	\end{align}
	\end{subequations}
where the specific relation between $\kappa$, $u$ and $\nabla\,u$ has not been made explicit, and the---a priori unknown---function $\varphi$ encodes the restriction of $u$ to the computational boundary  $\Gamma_h$. We can recover $\varphi$ following the method proposed by \cite{CoReGu2009} (in one dimension) and extended to higher dimensions by \cite{CoSo2012}. The idea consists of transferring the Dirichlet data $g$ from $\Gamma$ to $\Gamma_h$ along segments called \textit{transfer paths} by computing a line integral of the flux $\boldsymbol q$. 

To be precise, given $\boldsymbol{x}\in \Gamma_h$ and $\overline{\boldsymbol{x}} \in \Gamma$, equation \eqref{eq. G-S. numerical domain_a} can be integrated along the segment connecting them.  Lets denote by $\boldsymbol{t}(\mbf{x})$ the unit vector anchored at $\mbf{x}$ pointing towards $\overline{\boldsymbol{x}}$, and by $l(\mbf{x})$ the length of the segment connecting them. We then have the following representation for $\varphi$:
	\begin{equation}\label{def:varphi-P1}
	\varphi:= g(\overline{\mbf{x}})+ \int_0^{l(\mbf{x})} (\kap^{-1} \,\bq)(\mbf{x} + \boldsymbol{t}(\mbf{x})s) \cdot  \boldsymbol{t}(\mbf{x}) ds.
	\end{equation}

Note that $\varphi$ depends on the values of either $u$ or $\nabla\,u$ (through $\kappa^{-1}$), and $\boldsymbol q$ over the extended domain $\Omega_h^c$. As such, we should write $\varphi = \varphi(u,\nabla\,u,\boldsymbol q, \boldsymbol x)$ however, to keep notation simple, we will abstain from this and will write simply $\varphi$. In a similar fashion, $g(\overline{\boldsymbol x})$ is in fact a function of $\boldsymbol x$, since the point $\overline{\boldsymbol{x}}$ varies smoothly with $\boldsymbol x$. To avoid the use of cumbersome notation we will write either $g(\overline{\boldsymbol x})$ or simply $\g := g(\overline{\boldsymbol x}(\boldsymbol x))$.
%
% ====================================================
\subsection{Sobolev space notation}\label{sec:SpacesNorms}
% =====================================================
To denote spaces of functions we will make use of the standard notation and terminology from Sobolev space theory. Let $\mathcal{O}$ be a domain in 
$\mathbb{R}^{d}$, and $\Sigma$ be either a Lipschitz curve (if $d=2$) or surface (if $d=3$); for scalar-valued functions and non zero real numbers $s$, we will use the spaces $ H^s(\mathcal O)$ and $H^s(\Sigma)$ with their usual definition, whereas for the case $s=0$ we will write simply $L^{2}(\cO)$ and $ L^{2}(\Sigma)$. The spaces of vector-valued functions will be denoted in bold face, therefore $\bH^{s}(\cO) := [H^{s}(\cO)]^{d}$ and $\bH^{s}(\Sigma) := [H^{s}(\Sigma)]^{d}$.

The $L^2$ inner products for both scalar and vector-valued functions on volumes and surfaces will be denoted by $(\cdot,\cdot)_{\mathcal O}$ and $\pdual{\cdot,\cdot}_{\Sigma}$ respectively. The associated norms will be denoted by $\|\cdot\|_{s,\cO}$ and $\|\cdot\|_{s,\Sigma}$ and simply $\|\cdot\|_{\cO}$ for the case $s=0$. As is common, will we write $|\cdot|_{s,\cO}$ for the $\bH^{s}$ and  $H^{s}$-semi norms.

Given a triangulation $\mathcal T_h$ we will define  the following mesh-dependent inner products over elements and edges
	\begin{equation*}
	(\cdot, \cdot)_{\mc{T}_h} := \sum_{T\in \mc{T}_h} (\cdot, \cdot)_T , \qquad \pdual{\cdot, \cdot}_{\partial \mc{T}_h} :=\sum_{T\in \mc{T}_h} \pdual{\cdot, \cdot}_{\partial T} \quad  \text{ and } \quad  \langle\cdot, \cdot\rangle_{\Gamma_h} := \sum_{e\in \mc{E}_h^{\partial}} \langle\cdot, \cdot\rangle_e.
	\end{equation*}	
These inner products induce mesh-dependent norms that will be denoted, respectively, by
    \begin{equation*}
    \| \cdot \|_{\compD} := \left( \sum_{T\in \mc{T}_h} \| \cdot \|_T^2 \right)^{1/2}, \qquad \| \cdot \|_{\partial \mc{T}_h} := \left( \sum_{T\in  \mc{T}_h} \| \cdot \|_{\partial T}^2 \right)^{1/2} \quad  \text{ and } \quad   \| \cdot \|_{\Gamma_h} := \left( \sum_{e\in \mc{E}_h^{\partial}} \| \cdot \|_e^2 \right)^{1/2} .
    \end{equation*}
In the forthcoming analysis, the expression $a\lesssim b$ should be understood as meaning $a\leq C b$ where $C$ is a positive constant independent of $h$. 

For the discrete formulations that will be introduced in the next sections, we will make use of the following finite dimensional spaces of piece-wise polynomial functions
    \begin{subequations}\label{eq:PolynomialSpaces}
	\begin{align}
	\mbf{V}_h &:= \{\mbf{v}\in \boldsymbol L^2(\mc{T}_h) : \mbf{v}|_T \in [\md{P}_k(T)]^d, \ \forall \ T \in \mc{T}_h \}, \\
	W_h &:= \{w\in L^2(\mc{T}_h) : w|_T \in \md{P}_k(T), \ \forall \ T \in \mc{T}_h \}, \\
	M_h &:= \{\mu\in L^2(\mc{E}_h) : \mu|_T \in \md{P}_k(e), \ \forall \ e \in \mc{E}_h \},
	\end{align}
	\end{subequations}
where, $\md{P}_k(T)$ denotes the space of polynomials of degree at most $k$ defined in $T\in \mc{T}_h$. Similarly, $\md{P}_k(e)$ denotes the space of polynomials of degree at most $k$ defined over a face $e\in \mc{E}_h$. 
%
% ================================================
\section{Non-linearities of the form $\kappa(u)$}\label{sec:DependenceOnU}
% ================================================
%
% =======================================
\subsection{The HDG formulation}\label{sec:DependenceOnU-HDGform}
% =======================================
%
We will first consider the case when the coefficient $\kappa$ depends on the solution in the form 
\begin{alignat*}{6}
\kappa:\,& L^2(\Omega) && \longrightarrow & \; L^\infty(\overline{\Omega}) \\
& u && \longmapsto & \kappa(u).
\end{alignat*}
For the analysis, we will require the existence of positive constants $\underline{\kap}$ and $\overline{\kap}$ such that for all $u \in L^2(\Omega)$
    \begin{equation}\label{eq:assumptions_kappa}
     \underline{\kap} \leq \kap(u) \leq\overline{\kap} \quad  \, \text{ almost everywhere in } \Omega.
    \end{equation}	
Moreover, $\kap$ will be assumed to be Lipschitz-continuous on $L^2(\Omega)$, i.e, there exists $\tilde{L}>0$ such that
	\begin{equation}\label{eq:Lipschitz_k}
	\|\kap(u_1) - \kap(u_2)\|_{L^{\infty}(\Omega)} \leq \tilde{L} \|u_1-u_2\|_{L^2(\Omega)} \qquad \forall \, u_1,u_2 \in L^2(\Omega).
	\end{equation}
The two conditions above, together, imply the existence of constants $\what{L}$ and  $L$ such that 
	\begin{alignat}{6}
	\label{eq:Lipschitz_k1/2}
	\|\kap^{1/2}(u_1) - \kap^{1/2}(u_2)\|_{L^{\infty}(\Gamma)} \leq\,& \what{L} \|u_1-u_2\|_{L^2(\Omega)} &\qquad& \forall \, u_1,u_2 \in L^2(\Omega), \\[2ex]
	\label{eq:Lipschitz_k-1}
	\|\kap^{-1}(u_1) - \kap^{-1}(u_2)\|_{L^{\infty}(\overline{\Omega})} \leq\,& L \|u_1-u_2\|_{L^2(\Omega)} &\qquad& \forall \, u_1,u_2 \in L^2(\Omega).
	\end{alignat}
Note that all these assumptions imply the Lipschitz continuity of $\kappa, \kappa^{1/2}$, and $\kappa^{-1}$ on the subdomain $\Omega_h \subset \Omega$ with corresponding Lipschitz  constants equal to or smaller than those stated above.

Before introducing the discrete formulation we will recall here the mixed form \eqref{eq. Grad-Shaf. numerical domain}, but now we make explicit the dependence $\kappa = \kappa(u)$
	\begin{subequations}\label{eq:mixed}
	\begin{align}
	&& && \bq + \kap(u)\,  \nabla u &= 0 & &\text{ in } \Omega_h, &&\label{eq:mixed_a}\\
	&& && \nabla \cdot \bq  &= f(u) & &\text{ in }  \Omega_h, &&\label{eq:mixed_b}\\
	&& && u &= \varphi &  &\text{ on } \partial \Omega_h .&&\label{eq:mixed_c}
	\end{align}
	\end{subequations}
The boundary data $\varphi$ on the computational boundary $\Gamma_h$ is transferred according to \eqref{def:varphi-P1}. 

Taking an admissible triangulation $\mathcal T_h$ of the computational domain $\Omega_h$, the HDG discretization of \eqref{eq:mixed} reads: Find $(\bq_h, u_h, \what{u}_h) \in \mbf{V}_h\times W_h \times M_h$, such that
	\begin{subequations}\label{eq:HDG-P1}
	\begin{align}
	(\kap^{-1}(u_h)\,  \bq_h, \mbf{v})_{\mc{T}_h} - (u_h, \nabla \cdot \mbf{v})_{\mc{T}_h} + \langle\what{u}_h, \mbf{v} \cdot \mbf{n} \rangle_{\partial \mc{T}_h} &=  0,  \\
	-(\bq_h, \nabla w)_{\mc{T}_h} + \langle \what{\bq}_h \cdot \mbf{n},w \rangle_{\partial \mc{T}_h} &=  (f(u_h),w)_{\mc{T}_h}, \label{eq:HDG-P1_b} \\
	\label{eq:HDG-P1_c}
	\langle \what{u}_h,\mu\rangle_{\Gamma_h} &=  \langle \varphi_{h}(u_h),\mu\rangle_{\Gamma_h},  \\
	\langle \what{\bq}_h \cdot \mbf{n}, \mu \rangle_{\partial \mc{T}_h \setminus \Gamma_h} &= 0,
	\end{align}
for all $(\mbf{v}, w, \mu)\in \mbf{V}_h \times W_h \times M_h$. Here 
\[
    \what{\bq}_h\cdot \mbf{n} := \bq_h \cdot \mbf{n} +  \tau(u_h - \what{u}_h) \quad \textrm{on}\quad \partial \mc{T}_h
\]
with $\tau$ being a positive stabilization function, whose maximum will be denoted by $\overline{\tau}$. The approximate boundary condition $\varphi_h$ on the right hand side of \eqref{eq:HDG-P1_c} is given by the discrete counterpart of \eqref{def:varphi-P1}
	\begin{equation}\label{def:varphi_{h}-P1}
	\varphi_{h}(u_h)(\mbf{x}):=  g(\overline{\mbf{x}})+\int_0^{l(\mbf{x})} (\kap^{-1}(u_h)\,E_h\bq_h)(\mbf{x} + \boldsymbol{t}(\mbf{x})s) \cdot  \boldsymbol{t}(\mbf{x})) ds, \quad \textrm{for} \quad \mbf{x} \in \Gamma_h.
	\end{equation}
	\end{subequations}
In the definition above, we have used the extrapolation $E_h\bq_h$ due to the fact that the approximation $\bq_h$ is available only inside of the computational domain $\Omega_h$, but the transfer paths along which the integral is computed are defined over the complementary extended region $\Omega_h^c$.

% ===========================
\subsection{Well-posedness}\label{sec:well-posedness}
% ==========================
%assumptions
In this section we employ a Banach fixed-point argument to ensure the well-posedness of the discrete problem \eqref{eq:HDG-P1}. To that end we will define an operator $\mc{J}:  W_h  \to  W_h $ mapping $\zeta$ to the second component of the triplet $(\bq, u, \what{u})\in \mbf{V}_h\times W_h \times M_h$ satisfying, for all $(\mbf{v}, w, \mu)\in \mbf{V}_h \times W_h \times M_h$, the HDG system \eqref{eq:HDG-P1} where the source has been evaluated at $\zeta$, namely
	\begin{subequations}\label{eq:Fixed_point}
	\begin{align}
	(\kap^{-1}(\zeta)\, \bq, \mbf{v})_{\mc{T}_h} - (u, \nabla \cdot \mbf{v})_{\mc{T}_h} + \pdual{\what{u}, \mbf{v} \cdot \mbf{n}}_{\partial \mc{T}_h} &= 0,  \\
	-(\bq, \nabla w)_{\mc{T}_h} + \langle  \what{\bq} \cdot \mbf{n},w\rangle_{\partial \mc{T}_h}  &= (f(\zeta),w)_{\mc{T}_h}, \\
    \langle \what{u},\mu\rangle_{\Gamma_h} &=  \langle \varphi(\zeta),\mu\rangle_{\Gamma_h}, \\
	\langle \what{\bq} \cdot \mbf{n}, \mu\rangle_{\partial \mc{T}_h \setminus \Gamma_h} &= 0. 
	\end{align}
	\end{subequations}
Above, the term $\varphi(\zeta)$ corresponds to the boundary condition transferred to the computational domain by means of \eqref{def:varphi_{h}-P1}. The mapping $\mathcal J$ is well defined, as the linearized system \eqref{eq:Fixed_point} is uniquely solvable as proven in \cite{CoQiuSo2014}.

The main result of this section---that the mapping $\mathcal J$ defined above is a contraction---relies on the validity of a particular inequality---estimate \eqref{eq:KeyStep} below---but is otherwise a simple argument. Since the proof of \eqref{eq:KeyStep} requires a sequence of technical arguments, in the interest of clarity (we will first prove the main theorem assuming that the aforementioned inequality is valid. After having established the well posedness of the discrete problem, the reminder of the section will be devoted to verifying the validity of \eqref{eq:KeyStep}. This will be finally established in Lemma \ref{lem:StabilityS}, after a series of auxiliary results.
\begin{thm}[Well-posedness of the discrete problem]\label{thm:PuntoFijo}
Suppose that Assumptions \eqref{eq:Assumptions} are satisfied and that additionally
\begin{align*}
(\sqrt{3}\, \what{c} + \overline{\kap}^{1/2}\, R^{1/2} )\, \|l^{-1/2} \, \g\|_{\Gamma_h}\, \what{L}\, h^{1/2} < 1/4 , \\
\max\{ \what{c}^2h,1 \}\, L_{f}< 1/8,
\end{align*}
where $\widehat{c}$ is the constant given in Lemma \ref{lem:StabilityS}. Then $\mc{J}$ is a contraction operator.
\end{thm}

\begin{proof}
Let $\zeta_1, \zeta_2 \in W_h$ and define $u_1:=\mc{J}(\zeta_1)$ and $u_2:=\mc{J}(\zeta_2)$. Then $u_1$ and $u_2$ are the second components of solutions to \eqref{eq:Fixed_point} and Lemma \ref{lem:StabilityS} guarantees that
    \begin{align}
    \nonumber
    &\|\mc{J}(\zeta_1)-\mc{J}(\zeta_2)\|_{\compD} = \| u_1-u_2 \|_{\compD} \\
    \label{eq:KeyStep}
    &\quad \leq  4\, \max\{ \what{c}^2h,1 \}\,  \|f(\zeta_1) - f(\zeta_2)\|_{\compD} + 2\, (\sqrt{3}\, \what{c} + \overline{\kap}^{1/2}\, R^{1/2} )\, h^{1/2}\, \|(\kap^{1/2}(\zeta_1) - \kap^{1/2}(\zeta_2) ) \,  l^{-1/2} \,  \g \| _{\Gamma_h}.
    \end{align}
Then, applying the Lipschitz-continuity of $f$ and $\kap^{1/2}$---given respectively in \eqref{Lipschitz_F} and  \eqref{eq:Lipschitz_k}---we get   
    \begin{align*}
    &\quad \leq 4 \, \max\{ \what{c}^2h,1 \}\, L_{f}\,  \|\zeta_1 - \zeta_2\|_{\compD} + 2\, (\sqrt{3}\, \what{c} + \overline{\kap}^{1/2}\, R^{1/2} )\,  h^{1/2}\, \|\kap^{1/2}(\zeta_1) - \kap^{1/2}(\zeta_2)  \|_{L^{\infty}(\Gamma_h)} \, \| l^{-1/2} \,  \g \|_{\Gamma_h} \\
    &\quad \leq   4\, \max\{ \what{c}^2h,1 \}\, L_{f}\,  \|\zeta_1-\zeta_2\|_{\compD} + 2\, (\sqrt{3}\, \what{c} + \overline{\kap}^{1/2}\, R^{1/2} )\, \what{L}\, h^{1/2}\, \|\zeta_1 - \zeta_2 \|_{\compD} \, \| l^{-1/2} \,  \g \|_{\Gamma_h}.
\end{align*}
The result follows from the hypothesis for $\what{L}$ and $L_{f}$ . 
\end{proof}
Combined with a standard fixed-point argument, the theorem above guarantees the existence and uniqueness of the solution to the discrete HDG system. We will now direct our efforts to showing that inequality \eqref{eq:KeyStep} holds. To that end we will make use of the following auxiliary function and its properties listed in the lemma below---the proof of which can be found on \cite[Lemma 5.2]{CoQiuSo2014}.
\begin{lem}
Consider $\mbf{x}\in\Gamma_h$ and any smooth enough function $\mbf{v}$ defined in $T^e\cup T_{ext}^e$, and define
	\begin{equation}\label{def:delta}
	\delta_{\mbf{v}} (\mbf{x}) := \dfrac{1}{l(\mbf{x})} \int_0^{l(\mbf{x})} [\mbf{v}(\mbf{x} + \mbf{n}s) - \mbf{v}(\mbf{x}) ] \cdot \mbf{n}\, ds.
	\end{equation}
The following estimates hold for each $e\in \mc{E}_h^{\partial}$:
	\begin{subequations}\label{ineq:deltav}
	\begin{align}
	\label{ineq: est delta 1}
	\| l^{1/2} \, \delta_{\mbf{v}} \|_e &\leq \dfrac{1}{\sqrt{3}} \, r_e^{3/2} \, C_{ext}^e \, C_{inv}^e \, \|\mbf{v}\|_{T^e}  & &\forall \ \mbf{v} \in [\md{P}_k(T)]^d, \\
	\label{ineq: est delta 2}
	\| l^{1/2}\,\delta_{\mbf{v}} \|_e &\leq \dfrac{1}{\sqrt{3}} \, r_e \, \| h^{\perp} \partial_n \mbf{v} \cdot \mbf{n}\|_{T_{ext}^e} & &\forall \ \mbf{v} \in [H^1(T)]^d. \\
	\label{ineq: est delta 3}
	\|l^{1/2}\, \delta_{\mbf{v}}\|_{\infty} &\leq \dfrac{1}{\sqrt{3}} \, r_e \,  \sup_{\mbf{x}\in e} \|h_e^{\perp}\,\partial_n \mbf{v}\cdot \mbf{n} \|_{l(\mbf{x})} &&\forall \ \mbf{v} \in [H^1(T)]^d.
	\end{align}
	\end{subequations}
\end{lem}
For the subsequent analysis, we will also make use of the following norm. Given$(\mbf{v}, w, \mu)\in \mbf{V}_h \times W_h \times M_h$ and $\xi\in  M_h$ we define
    \begin{align}\label{def:triple_norm-P1}
          \triple{(\mbf{v},w,\mu)}_{\xi} &:= \left( \| \kap^{-1/2}(\xi) \mbf{v} \|_{\compD}^2 + \|\tau^{1/2} \, w\|_{\partial \mc{T}_h}^2 + \| \kap^{1/2}(\xi)\,  l^{-1/2} \, \mu(\xi) \|_{\Gamma_h}^2 \right)^{1/2}.
    \end{align}	
The proof of \eqref{eq:KeyStep} requires considering the solution $\phi$ to an auxiliary problem (c.f. \eqref{eq:DualProblem} ) and using a duality argument that connects a stability estimate for $\phi$ with our variables of interest. This will be done in Lemma \ref{lem:StabilityS}. Lemmas \ref{lem:EstimateVarphi} and \ref{lem:EstimateNormH} below establish estimates relating the norm \eqref{def:triple_norm-P1} of the solution $(\boldsymbol q, u, \widehat{u})$ to problem data that will be used in the final step of Lemma \ref{lem:StabilityS}.
\begin{lem}\label{lem:EstimateVarphi}
Let $\varphi$ be the transferred boundary condition appearing in \eqref{eq:Fixed_point} and suppose that assumptions \eqref{eq:Assumptions}  are satisfied. It holds 
	\begin{align*}
	\langle \varphi(\zeta) , \delta_{\bq }\rangle_{\Gamma_h} &\leq \dfrac{1}{6} \, \|\kap^{1/2}(\zeta)\, l^{-1/2}\, \varphi(\zeta)\|^2_{\Gamma_h} + \dfrac{1}{2}\,  \|\kap^{-1/2}(\zeta)\,  \bq\|_{\compD}^2, \\[2ex]
	\langle \varphi(\zeta), \tau(u - \what{u})\rangle_{\Gamma_h} &\leq \dfrac{1}{6} \|\kap^{1/2}(\zeta)\, l^{-1/2}\, \varphi(\zeta)\|^2_{\Gamma_h} + \dfrac{1}{2} \|\tau^{1/2} (u - \what{u})\|^2_{\partial \mc{T}_h}, \\[2ex]
   \langle\varphi(\zeta), \kap(\zeta) \, l^{-1} \, \g\rangle_{\Gamma_h} &\leq \dfrac{1}{6} \|\kap^{1/2}(\zeta)\, l^{-1/2}\, \varphi(\zeta)\|^2_{\Gamma_h} + \dfrac{3}{2} \|\kap^{1/2}(\zeta) \, l^{-1/2}\,  \g\|^2_{\Gamma_h},
	\end{align*}
where $\g(\mbf{x})=g(\overline{\mbf{x}}(\mbf{x})) \,\forall\,  \mbf{x}\in \Gamma_h$.
\end{lem}
\begin{proof}
The first inequality is obtained after applying Young's inequality and estimate \eqref{ineq: est delta 1}, whereas the second inequality follows from assumption \eqref{eq:S2} and \eqref{eq:S3}. The third inequality follows from Young's inequality exclusively.
\end{proof}
\begin{lem}\label{lem:EstimateNormH} 
If Assumptions \eqref{eq:Assumptions} hold, then
	\begin{align*}
\triple{(\bsy{q}, u-\what{u}, \varphi)}_{\zeta}^2 &\leq 2\|f(\zeta)\|_{\compD} \|u\|_{\compD}  + 3 \|\kap^{1/2}(\zeta)\, l^{-1/2} \, \g \|_{\Gamma_h}^2.
	\end{align*}
\end{lem}
\begin{proof}
Let $\zeta \in W_h$ and $u = \mathcal J(\zeta) \in W_h$. Since $u$ defined this way is the solution to the discrete system \eqref{eq:Fixed_point}, then testing \eqref{eq:Fixed_point} with
    \begin{equation*}
    \mbf{v}= \bq, \quad w=u,\quad \mu:= \left\{ \begin{array}{rcl}
	 - \what{\bq} \cdot \mbf{n} &  &\text{on } \Gamma_h ,\\
	 - \what{u}&  &\text{on } \partial \mc{T}_h \setminus \Gamma_h
	 \end{array} 
	 \right.
    \end{equation*}
we deduce that
\begin{equation}\label{eq:eq}
\|\kap^{-1/2}(\zeta)\, \bq\|^2_{\compD} + \| \tau^{1/2} \,(u- \what{u}) \|^2_{\partial \mc{T}_h}  =  - \pdual{\varphi(\zeta),\what{\bq}\cdot \mbf{n}}_{\Gamma_h} + (f(\zeta),u)_{\mc{T}_h}.
\end{equation}
On the other hand, we can use the definition of $\varphi$ and $\delta_{\boldsymbol q}$ (cf. \eqref{def:varphi_{h}-P1} and \eqref{def:delta}) to show that
\[
\what{\bq}\cdot \mbf{n} = \kap(\zeta)\,l^{-1}  (\varphi(\zeta)-\g) - \delta_{\bq} + \tau(u - \what{u}).
\]
Substituting the expression for $\what{\bq}\cdot \mbf{n}$ above in \eqref{eq:eq} , we obtain
    \begin{equation*}
    \begin{array}{rl}
    \|\kap^{-1/2}(\zeta)\, \bq\|^2_{\compD}  &+\, \| \tau^{1/2}\, (u- \what{u}) \|^2_{\partial \mc{T}_h}  +  \|\kap^{1/2}(\zeta)\, l^{-1/2}\,  \varphi(\zeta) \|_{\Gamma_h}^2 \\[2ex]
    & = \triple{(\mbf{q},u-\widehat{u},\varphi)}_{\zeta}^2 \\[2ex]
    &\leq |\langle \varphi(\zeta), \kap (\zeta)\,l^{-1}\, \g\, \rangle_{\Gamma_h} |  + | \langle \varphi(\zeta), \delta_{\bq} \rangle_{\Gamma_h} | + |\langle \varphi(\zeta), \tau(u - \what{u})\rangle_{\Gamma_h}| + |(f(\zeta),u)_{\mc{T}_h}|. \\
    \end{array}
    \end{equation*}
Now, using Lemma \ref{lem:EstimateVarphi} to estimate the first three terms in the right hand of this expression we obtain 
    \begin{equation*}
   \dfrac{1}{2} \, \triple{(\bsy{q}, u-\what{u}, \varphi)}_{\zeta}^2  \leq \|f(\zeta)\|_{\compD} \|u\|_{\compD}  + \dfrac{3}{2} \|\kap^{1/2}(\zeta)\, l^{-1/2} \, \g \|_{\Gamma_h}^2, \\
    \end{equation*}
whereupon the proof is concluded. 
\end{proof}
For the following result, we will make use of the properties of the HDG projectors $\boldsymbol\Pi_{\boldsymbol V}$ and $\Pi_{W}$ onto the discrete spaces $\boldsymbol V_h$ and $W_h$. This projection was first introduced in \cite{CoGoSa2010} and we include its definition and main properties in the Appendix \ref{sec:HDGprojection}. The $L^2$ projector onto the space $M_h$ will be denoted by $P_M$, while $Id_M$ will denote the identity on $M_h$.

\begin{lem}\label{lem:StabilityS} Suppose that Assumptions \eqref{eq:Assumptions} and the regularity \eqref{regularity dual problem} are satisfied. Then, there exists $\what{c}>0$, independent of $h$ such that 
	\begin{equation}
	\|u\|_{\compD} \leq 4\, \max\{ \what{c}^2h,1 \}\,  \|f(\zeta)\|_{\compD}+  2\, \left(\sqrt{3}\, \what{c} + \overline{\kap}^{1/2}\, R^{1/2} \right)\,h^{1/2}\, \|\kap^{1/2}(\zeta) \,  l^{-1/2} \,  \g \| _{\Gamma_h}.
	\end{equation}
\end{lem}

\begin{proof}
Consider $\Theta\in L^2(\Omega)$ and let $\phi$ and $\psi$ be the solutions to the dual problem \eqref{eq:DualProblem} associated to $\Theta$. If we define 
    \begin{equation*}
    \md{T}_{\bq} := ( \kap^{-1}(\zeta) \bq, \bsy{\Pi}_{\mbf{V}} \bsy{\phi} - \bsy{\phi})_{\mc{T}_h}, \quad 
    \md{T}_{u} := \langle \what{u}, P_M(\bsy{\phi} \cdot \mbf{n}) \rangle_{\Gamma_h} - \langle\what{\bq} \cdot \mbf{n} , \Pi_W \psi\rangle_{\Gamma_h} \quad \text{and} \quad \md{T}_{f} := (f(\zeta), \Pi_W \psi)_{\mc{T}_h}, 
    \end{equation*}
it is possible to verify that
	\begin{equation}\label{eq:dual_arg}
	(u, \Theta)_{\mc{T}_h} = \md{T}_{\bq} + \md{T}_{f} + \md{T}_{u}.
	\end{equation}	 
The terms $\md{T}_{\bq}$ and $\md{T}_{f}$ appearing on the expression above can be easily estimated by
    \begin{align}\label{eq:bounds_Tq-Tf}
    |\md{T}_{\bq}| &\lesssim \underline{\kap}^{-1/2}\, h \|\kap^{-1/2}(\zeta)\, \bq\|_{\compD} \, \|\Theta\|_{\Omega}, \quad \text{and} \quad |\md{T}_{f} | \lesssim  \|f(\zeta)\|_{\compD}\, \|\Theta\|_{\Omega}.
    \end{align}
In order to bound the final term of the decomposition of $(u, \Theta)_{\mc{T}_h}$, we rewrite $\md{T}_{u} = \sum_{i=1}^5 \md{T}_{u}^i$ where 
    \begin{equation*}
    \begin{array}{lcl}
     \md{T}_{u}^1 := - \langle\kap(\zeta) l^{-1} \varphi(\zeta) ,\psi + l \partial_n \psi\rangle_{\Gamma_h}, & & \md{T}_{u}^4 := - \langle\tau(u - \what{u}) ,P_M \psi\rangle_{\Gamma_h}, \\[2ex]
    \md{T}_{u}^2 := \langle\kap(\zeta) \varphi(\zeta), (P_M - Id_M)\partial_n \psi \rangle_{\Gamma_h}, &&
	\md{T}_{u}^5 := \langle\kap(\zeta)\,  l^{-1} \, \g, \psi\rangle_{\Gamma_h}, \\[2ex]
	\md{T}_{u}^3 := \langle \delta_{\bq}, \psi \rangle_{\Gamma_h}.&& 
	\end{array}
    \end{equation*}
It is not hard, if cumbersome, to verify that for the terms above the following estimates hold
    \begin{alignat*}{10}
    |\md{T}_{u}^1| &\lesssim   \overline{\kap}^{1/2}\, R\,h\, \|\kap^{1/2}(\zeta) \,l^{-1/2}\, \varphi(\zeta) \|_{\Gamma_h} \|\Theta \|_{\Omega}, & \qquad &&
    |\md{T}_{u}^2| &\lesssim \overline{\kap}^{1/2}\,R^{1/2} \,h\, \|\kap^{1/2}(\zeta) l^{-1/2} \varphi(\zeta) \|_{\Gamma_h} \|\Theta \|_{\Omega},\\[2ex]
    |\md{T}_{u}^3| &\lesssim  \overline{\kap}^{1/2} \, R^2\, h^{1/2} \|\kap^{-1/2}(\zeta)\bq\|_{\compD}\|\Theta\|_{\Omega}, & \qquad &&
    |\md{T}_{u}^4| & \lesssim  \overline{\tau}^{1/2}\,  R\, h\,  \|\tau^{1/2}(u -\what{u})\|_{\partial \mc{T}_h} \|\Theta\|_{\Omega}, \\[2ex]
    |\md{T}_{u}^5| & \lesssim \overline{\kap}^{1/2}\,(Rh)^{1/2}\, \|\kap^{1/2}(\zeta) \,  l^{-1/2} \,  \g \| _{\Gamma_h} \|\Theta\|_{\Omega}. & \qquad && 
    \end{alignat*}
Taking $\Theta=u$ in  \eqref{eq:dual_arg} and combining all of the above estimates with \eqref{eq:bounds_Tq-Tf}, we obtain
\[
\|u\|_{\compD} \leq \what{c}\, h^{1/2} \triple{(\bsy{\sigma}, u-\what{u}, \varphi)}_{\zeta} + \overline{\kap}^{1/2}\,  (Rh)^{1/2} \|\kap^{1/2}(\zeta) \,  l^{-1/2} \,  \g \| _{\Gamma_h} + \|f(\zeta)\|_{\compD},
\]
where $\what{c}:=C\,  \max\{ \underline{\kap}^{-1/2},  \overline{\kap}^{1/2} R,  \overline{\kap}^{1/2} R^2,  \overline{\kap}^{1/2} R^{1/2}, \overline{\tau}^{1/2} R \}$, and $C>0$ is the constant hidden in the symbol $\lesssim$. Then, applying Lemma \ref{lem:EstimateNormH}, we get
    \begin{align*}
	\|u\|_{\compD} &\leq \quad  \what{c} \, h^{1/2}  \left( \sqrt{2}\|f(\zeta)\|_{\compD}^{1/2} \|u\|_{\compD}^{1/2}  + \sqrt{3} \|\kap^{1/2}(\zeta)\, l^{-1/2} \, \g \|_{\Gamma_h} \right) \\
	& \quad + \overline{\kap}^{1/2}\,  (Rh)^{1/2} \|\kap^{1/2}(\zeta) \,  l^{-1/2} \,  \g \| _{\Gamma_h}  + \|f(\zeta)\|_{\compD} \\[2ex]
		\|u\|_{\compD}&\leq 4\, \max\{ \what{c}^2h,1 \}\,  \|f(\zeta)\|_{\compD}+  2\, (\sqrt{3}\, \what{c} + \overline{\kap}^{1/2}\, R^{1/2} )\,h^{1/2}\, \|\kap^{1/2}(\zeta) \,  l^{-1/2} \,  \g \| _{\Gamma_h},
	\end{align*}
with  which the proof is concluded.
\end{proof}
%
% =============================================
\subsection{\textit{A prior} error analysis}\label{sec:Prob1Apriori}
% ===========================================

We now provide the {\it a priori} error bounds for the discretization error. The main results of the section are Theorem \ref{eq:HDG_error-P1} and Corollary \ref{corol:estimate_error-P1} immediately after it. As we will see, several of the results leading to the main presented in this section can be proven by using similar arguments to those of Section \ref{sec:well-posedness} and we will omit some of the arguments. The analysis will be performed by decomposing the approximation errors in two components using the properties of the HDG projection (see Appendix \ref{sec:HDGprojection}). The \textit{projection of the errors} is defined as
\[
\bsy{\varepsilon}^{\bq}:= \bsy{\Pi}_{\mbf{V}}\bq - \bq_h \quad \text{ and }  \quad \varepsilon^u:= \Pi_W u - u_h,
\]
and the \textit{error of the projections} are given by
\[
\mbf{I}^{\bq}:= \bq-\bsy{\Pi}_{\mbf{V}}\bq \quad \text{ and } \quad I^u:=u - \Pi_W u.
\]
This allows to express the approximation errors as
\[
\bq-\bq_h = \bsy{\varepsilon}^{\bq}+\mbf{I}^{\bq} \quad \text{ and } \quad u-u_h =  \varepsilon^u + I^u.
\]
In addition, recalling that $P_M$ is the $L^2$ projection into $M_h$, we define the projection error for the hybrid unknown $\widehat{u}_h$ as $\varepsilon^{\what{u}}:= P_M u - \what{u}_h$. The $L^2$-projection of the error for the numerical flux on $\partial \mc{T}_h$ can be expressed as $\errortq \cdot \mbf{n} = \errorq \cdot \mbf{n} +  \tau(\erroru - \errortu)$. It is not difficult show that $( \bsy{\varepsilon}^{\bq},\varepsilon^u,\varepsilon^{\what{u}})$ belongs to $\mbf{V}_h\times W_h \times M_h$ and satisfies 
	\begin{subequations}\label{eq:HDG_error-P1}
	\begin{align}
	(\kap^{-1}(u_h) \errorq, \mbf{v})_{\mc{T}_h} - (\erroru, \nabla \cdot \mbf{v})_{\mc{T}_h} + \langle \varepsilon^{\what{u}}, \mbf{v} \cdot \mbf{n}\rangle_{\partial \mc{T}_h} =&   -(\kap^{-1}(u)\mbf{I}^{\bq},v )_{\mc{T}_h} \nonumber \\
	& - \left((\kap^{-1}(u) - \kap^{-1}(u_h))\Pi_{\mbf{V}}\bq,v\right)_{\mc{T}_h} , \label{error projection 1}  \\
	-(\bsy{\varepsilon}^{\bq}, \nabla w)_{\mc{T}_h} + \langle \bsy{\varepsilon}^{\what{\bq}} \cdot \mbf{n},w\rangle_{\partial \mc{T}_h} =&\, (f(u) - f(u_h), w)_{\mc{T}_h}, \\
	\langle \varepsilon^{\what{u}},\mu \rangle_{\Gamma_h} =&\,  \langle \varphi(u) - \varphi_{h}(u_h) ,\mu \rangle_{\Gamma_h},  \\
	\langle \bsy{\varepsilon}^{\what{\bq}} \cdot \mbf{n}, \mu\rangle_{\partial \mc{T}_h \setminus \Gamma_h} =&\, 0,
	\end{align}
	\end{subequations}
for all $(\mbf{v}, w, \mu)\in \mbf{V}_h \times W_h \times M_h$.

To try and keep the notation compact, we will define the following two quantities involving only the errors in the projections $\boldsymbol I^{\boldsymbol v}$, $\boldsymbol I^{\boldsymbol q}$, and $I^{u}$ measured in the three relevant domains $\Omega_h$, $\Omega_h^c$ and $\Gamma_h$ 
    \begin{subequations}\label{def:lambda}
    \begin{eqnarray}
    \projerrorq &: =&  \left( \| \mbf{I}^{\bq}\|^2_{\compD} +  \| h^{\perp} \partial_n (\mbf{I}^{\bq} \cdot \mbf{n})\|^2_{\Omega_h^c}  + \| (h^{\perp})^{1/2} \mbf{I}^{\bq} \cdot \mbf{n})\|^2_{\Gamma_h} \right)^{1/2}, \label{def:lambda_q} \\
    \projerroru &: =&  \left( \|(h^{\perp})^{1/2} I^u \|_{\Gamma_h}  + \|I^u\|_{\compD} \right)^{1/2}. \label{def:lambda_u}
    \end{eqnarray}
    \end{subequations}
We note that, as pointed out in \cite{SaSaSo2019,CoGoSa2010} by using the properties of the projectors and scaling arguments, if $\bq\in \mbf{H}^{k+1}(\Omega)$, $u\in H^{k+1}(\Omega)$ and $\tau$ is of order one, then $\projerrorq$ and $\projerroru$ are of order $h^{k+1}$. As stated in the theorem below, these quantities are in fact the key to estimating the approximation error of the method
\begin{thm}\label{thm:estimate_error-P1}
If $L$ is small enough, the regularity \eqref{regularity dual problem} holds and the discrete spaces are of polynomial degree $k\geq 1$, then there exists $h_0>0$ such that, for all $h\leq h_0$, we have 
	\begin{eqnarray}\label{eq:estimate_error-P1}
    \triple{ (\errorq ,\erroru - \errortu, \varphi - \varphi_{h})}_{u_h}^2  \lesssim \projerrorq^2 + \projerroru^2. 
	\end{eqnarray}
\end{thm}
The proof of this result will follow straightforwardly from lemmas \ref{lem:norm_error} and \ref{lem:estimation E_u} below. 
Before setting out to prove these two lemmas (and therefore the theorem above) we first state the convergence order of the method---the main result of the section---which thanks to the remark made just above Theorem \ref{thm:estimate_error-P1} follows as a corollary.
\begin{crl}[Order of convergence]\label{corol:estimate_error-P1}
Suppose that the assumptions of Theorem \ref{thm:estimate_error-P1} hold. If, in addition, $u\in H^{k+1}(\Omega)$ and $\bq \in  \mbf{H}^{k+1}(\Omega)$, then
\begin{eqnarray*}
\|\bq-\bq_h\|_{\Omega} + \|u-u_h\|_{\Omega} \leq C h^{k+1} \left( |u|_{k+1,\Omega} + |\bq|_{k+1,\Omega}  \right).
\end{eqnarray*}
\end{crl}
Having stated the main results of the section, we now set out to prove the two lemmas leading to Theorem \ref{thm:estimate_error-P1}. The first part of the analysis will require using an energy argument on the error equations \eqref{eq:HDG_error-P1} and a meticulous study of the error contribution due to the transferred boundary conditions $\varphi_h(u_h)$. This will be done in the following
\begin{lem}\label{lem:norm_error}
There exist positive constants, independent of $h$, such that
	\begin{align}\label{eq:norm_error}
	\triple{ (\errorq ,\erroru - \errortu, \varphi - \varphi_{h}) }_{{u_h}}^2 \leq 12\, \max \{ C_1\, h, C_2\} L^2  \, \left( \|\erroru\|_{\compD}^2 + \| I^{u} \|_{\compD}^2  \right)  + C_3\,  \projerrorq^2 .
	\end{align}
\end{lem}

\begin{proof}
Starting from \eqref{eq:HDG_error-P1} and letting
\[
\mbf{v}= \errorq, \;\text{ and }\;  w=\erroru\; \text{ in } \mc{T}_h, \; \text{ and }\; \mu:= \left\{ \begin{array}{rcl} - \, \errortq \cdot \mbf{n} &  &\text{on } \Gamma_h \\ 
-  \errortu &  &\text{on } \partial \mc{T}_h \setminus \Gamma_h   \end{array}  \right.
\]
it follows that
    \begin{align}
    \label{eq:norm_error_1}
    &\|\kap^{-1/2}(u_h)\, \errorq\|^2_{\compD} + \|\tau^{1/2}(\erroru - \errortu) \|^2_{\partial \mc{T}_h} = \\[2ex]
    \nonumber
    & -\!(\kap^{-1}\!(u)\mbf{I}^{\bq}, \errorq )_{\mc{T}_h}\! 
     - \! ((\kap^{-1}(u) - \kap^{-1}\!(u_h))\Pi_{\mbf{V}}\bq,\errorq)_{\mc{T}_h} \! +\!  (f(u) - f(u_h), \erroru)_{\mc{T}_h}\! - \!\langle \varphi(u) - \varphi_{h}(u_h) ,  \errortq\cdot \mbf{n} \rangle_{\Gamma_h}.
    \end{align}
We will now manipulate the final term in the expression above to include a term involving the norm of the difference $\varphi(u) - \varphi_{h}(u_h)$, thus allowing us to estimate the transfer error. Using definitions of $\varphi_{h}$ and $\varphi$ (cf. \eqref{def:varphi-P1} and \eqref{def:varphi_{h}-P1} respectively), as well as the definition of $\delta_{\boldsymbol q}$ it follows that
\[
\varphi(u)-\overline{g} = \kappa^{-1}(u)\ell\left(\delta_{\boldsymbol q} +\boldsymbol q\cdot\boldsymbol n\right) \quad \text{ and } \quad \varphi_h(u_h)-\overline{g} = \kappa^{-1}(u_h)\ell\left(\delta_{\boldsymbol q_h} +\boldsymbol q_h\cdot\boldsymbol n\right).
\]
Subtracting the second expression from the first one and adding zero in the form of $\pm\kappa^{-1}(u_h)\left(\delta_{\boldsymbol q}-\boldsymbol q\cdot\boldsymbol n\right)$ it is possible to express the difference as
\begin{align}
\nonumber
\varphi(u) - \varphi_{h}(u_h) =\,& \ell\left(\kappa^{-1}(u)-\kappa^{-1}(u_h)\right) (\delta_{\boldsymbol q} + \boldsymbol q\cdot\boldsymbol n) + \ell\,\kappa^{-1}(u_h)\left(\delta_{\boldsymbol q - \boldsymbol q_h} + (\boldsymbol q - \boldsymbol q_h)\cdot\boldsymbol n\right)\\
\nonumber
=\,& \ell\left(\kappa^{-1}(u)-\kappa^{-1}(u_h)\right)\Big(\delta_{\boldsymbol q} + \boldsymbol q\cdot\boldsymbol n\Big) + \ell\,\kappa^{-1}(u_h)\left(\delta_{\boldsymbol\varepsilon^{\boldsymbol q}} + \delta_{\boldsymbol I^{\boldsymbol q}} + \left(\boldsymbol\varepsilon^{\widehat{\boldsymbol q}} + \boldsymbol I^{\boldsymbol q}\right)\cdot\boldsymbol n -\tau\left(\varepsilon^{u}-\varepsilon^{\widehat{u}}\right)\right)\\
\nonumber
=\,& \ell\left(\kappa^{-1}(u)-\kappa^{-1}(u_h)\right)\Big(\delta_{\boldsymbol I^{\boldsymbol q}} + \delta_{\Pi_v \boldsymbol q} + \big(\boldsymbol I^{\boldsymbol q} + \Pi_{\boldsymbol V} \boldsymbol q\big)\cdot\boldsymbol n\Big) \\
\label{eq:transfererror}
& + \ell\kappa^{-1}(u_h)\left(\delta_{\boldsymbol\varepsilon^{\boldsymbol q}} + \delta_{\boldsymbol I^{\boldsymbol q}} + \left(\boldsymbol\varepsilon^{\widehat{\boldsymbol q}} + \boldsymbol I^{\boldsymbol q}\right)\cdot\boldsymbol n -\tau\left(\varepsilon^{u}-\varepsilon^{\widehat{u}}\right)\right),
\end{align}
where the first equality comes from the substitutions
\[
(\boldsymbol q-\boldsymbol q_h)\cdot\boldsymbol n = \left(\boldsymbol\varepsilon^{\widehat{\boldsymbol q}} + \boldsymbol I^{\boldsymbol q}\right)\cdot\boldsymbol n - \tau\left(\varepsilon^{u}-\varepsilon^{\widehat{u}}\right), \quad \text{and} \quad \delta_{\boldsymbol q + \boldsymbol q_h} = \delta_{\boldsymbol q}+\delta_{\boldsymbol q_h},
\]
while the second one is obtained by replacing $\boldsymbol q = \boldsymbol I^{\boldsymbol q} + \Pi_{\boldsymbol V}  \boldsymbol q$. The expression \eqref{eq:transfererror} allows us to write the term $\boldsymbol\varepsilon^{\widehat{\boldsymbol q}}\cdot\boldsymbol n$ in terms of the transfer error
    \begin{align*}
    \errortq\cdot \mbf{n} &= \kap(u_h)\,  l^{-1}\, (\varphi(u) - \varphi_{h}(u_h)) - \kap(u_h)\, (\kap^{-1}(u) - \kap^{-1}(u_h))(\delta_{\mbf{I}^{\bq}} + \delta_{\Pi_{\mbf{V}} \bq} + \mbf{I}^{\bq} \cdot \mbf{n} + \Pi_{\mbf{V}} \bq \cdot \mbf{n}) \\
    &\quad - \delta_{\errorq} - \delta_{\mbf{I}^q} - \mbf{I}^{\bq} \cdot \mbf{n} + \tau (\erroru - \errortu).
    \end{align*}
Substituting this expression back into \eqref{eq:norm_error_1} and rearranging terms, it follows that 
    \begin{equation}\label{eq:norm_error_2}
    \begin{array}{rl}
    \|\kap^{-1/2}(u_h) \, \errorq\|^2_{\compD} +& \|\tau^{1/2}(\erroru - \errortu) \|^2_{\partial \mc{T}_h} + \| \kap^{1/2}(u_h)  \, l^{-1/2} \, (\varphi(u) - \varphi_{h}(u_h) )\|^2_{\Gamma_h}  \\[2ex]
    &\quad \leq  | (\kap^{-1}(u) \mbf{I}^{\bq}, \errorq )_{\mc{T}_h}| + | ((\kap^{-1}(u) - \kap^{-1}(u_h))\Pi_{\mbf{V}}\bq,\errorq)_{\mc{T}_h}| +  |\md{T}^{f}| + |\md{T}_{\varphi}|,
    \end{array}
    \end{equation}
with $\md{T}^{f}:= (f(u)-f(u_h), \erroru)_{\mc{T}_h}$ and $\md{T}_{\varphi} := \sum_{i=1}^8{|\mr{T}_{\varphi}^i}|$, where
\small{
    \begin{equation*}
    \begin{array}{rclcrcl}
    \md{T}_{\varphi}^1 &:=& \langle \varphi(u) - \varphi_{h}(u_h), \delta_{\errorq} \rangle_{\Gamma_h} && \md{T}_{\varphi}^5 &:=& \langle \varphi(u) - \varphi_{h}(u_h), \kap(u_h)\, (\kap^{-1}(u) - \kap^{-1}(u_h))\, \delta_{\Pi_{\mbf{V}}\bq} \rangle_{\Gamma_h}  \\
    \md{T}_{\varphi}^2 &:=& - \langle \varphi(u) - \varphi_{h}(u_h), \tau (\erroru - \errortu) \rangle_{\Gamma_h} && \md{T}_{\varphi}^6 &:=& \langle \varphi(u) - \varphi_{h}(u_h),  \kap(u_h)\, (\kap^{-1}(u) - \kap^{-1}(u_h))\, \delta_{\mbf{I}^{\bq}}  \rangle_{\Gamma_h}  \\
    \md{T}_{\varphi}^3 &:=& \langle \varphi(u) - \varphi_{h}(u_h), \mbf{I}^{\bq} \cdot \mbf{n} \rangle_{\Gamma_h} && \md{T}_{\varphi}^7 &:=& \langle \varphi(u) - \varphi_{h}(u_h),  \kap(u_h)\, (\kap^{-1}(u) - \kap^{-1}(u_h))\, \mbf{I}^{\bq} \cdot \mbf{n} \rangle_{\Gamma_h}  \\
    \md{T}_{\varphi}^4 &:=& \langle \varphi(u) - \varphi_h(u_h),\delta_{\mbf{I}^{\bq}} \rangle_{\Gamma_h} && \md{T}_{\varphi}^8 &:=& \langle \varphi(u) - \varphi_{h}(u_h),  \kap(u_h)\,(\kap^{-1}(u) - \kap^{-1}(u_h))\, \Pi_{\mbf{V}}\bq \cdot \mbf{n} \rangle_{\Gamma_h}.
    \end{array}
    \end{equation*}
}
To determine upper bounds the terms in the right hand side of \eqref{eq:norm_error_2}, we will make use of Young's inequality, the Lipschitz continuity of $f$ and $\kappa^{-1}$ and the fact that $\|v\|_{L^2(e)} \lesssim h_e^{1/2} \, \|v\|_e, \, \forall \, e\in \mc{E}_h^{\partial}, \forall \, v\in \md{P}_k(e)$. A combination of these with arguments similar as those in \cite[Lemma 5]{SaSaSo2019} results in the following
    \begin{align*}
    |\md{T}_{\varphi}^1| &\leq \dfrac{1}{2\delta_1} \| \kap^{1/2}(u_h)\, l^{-1/2}( \varphi(u) - \varphi_{h}(u_h))\|^2_{\Gamma_h} + \dfrac{\delta_1}{6} \, \| \kap^{-1/2}(u_h) \errorq \|^2_{\compD}, \\ 
    |\md{T}_{\varphi}^2| &\leq \dfrac{1}{2\delta_1} \| \kap^{1/2}(u_h)\, l^{-1/2}( \varphi(u) - \varphi_{h}(u_h))\|^2_{\Gamma_h}  + \dfrac{\delta_1}{6}\,\| \tau^{1/2}(\varepsilon^u - \varepsilon^{\what{u}})\|^2_{\partial \mc{T}_h}, \\
    |\md{T}_{\varphi}^3| &\leq \dfrac{1}{2\delta_2} \| \kap^{1/2}(u_h)\, l^{-1/2}( \varphi(u) - \varphi_{h}(u_h))\|^2_{\Gamma_h} + \dfrac{\delta_2}{2} R \underline{\kap}^{-1} \,  \| (h^{\perp})^{1/2} \mbf{I}^{\bq} \cdot \mbf{n})\|^2_{\Gamma_h} , \\
    |\md{T}_{\varphi}^4| &\leq \dfrac{1}{2\delta_2} \| \kap^{1/2}(u_h)\, l^{-1/2}( \varphi(u) - \varphi_{h}(u_h))\|^2_{\Gamma_h}  + \dfrac{\delta_2}{6} \, \underline{\kap}^{-1} \, \max_{e\in \mc{E}_h^\partial} \{r_e^2\} \,  \| h^{\perp} \partial_n (\mbf{I}^{\bq} \cdot \mbf{n})\|^2_{\compD^c} , \\
    |\md{T}_{\varphi}^5| &\leq \dfrac{1}{2\delta_2} \| \kap^{1/2}(u_h)\, l^{-1/2}( \varphi(u) - \varphi_{h}(u_h))\|^2_{\Gamma_h}  + \dfrac{\delta_2}{6}\,  \overline{\kap} \, \| \Pi_{\mbf{V}}\bq \|^2_{L^{\infty}(\compD)} \,h \, L^2 \, \left( \|\erroru\|_{\compD} + \| I^{u} \|_{\compD}  \right) ^2, \\
    |\md{T}_{\varphi}^6| &\leq \dfrac{1}{2\delta_2} \| \kap^{1/2}(u_h)\, l^{-1/2}( \varphi(u) - \varphi_{h}(u_h))\|^2_{\Gamma_h}  +  \dfrac{2\delta_2}{3}\,  \overline{\kap} \, \underline{\kap}^{-2}  \, \max_{e\in \mc{E}_h^\partial} r_e^2\,  \| h^{\perp} \partial_n (\mbf{I}^{\bq} \cdot \mbf{n})\|^2_{\compD^c} ,\\
    |\md{T}_{\varphi}^7| &\leq \dfrac{1}{2\delta_2} \| \kap^{1/2}(u_h)\, l^{-1/2}( \varphi(u) - \varphi_{h}(u_h))\|^2_{\Gamma_h}  + 2\, \delta_2 \, R\, \overline{\kap} \, \underline{\kap}^{-2} \,  \| (h^{\perp})^{1/2} \mbf{I}^{\bq} \cdot \mbf{n})\|^2_{\Gamma_h}  ,\\
    |\md{T}_{\varphi}^8| &\leq \dfrac{1}{2\delta_2} \| \kap^{1/2}(u_h)\, l^{-1/2}( \varphi(u) - \varphi_{h}(u_h))\|^2_{\Gamma_h}  + \dfrac{\delta_2}{2} \overline{\kap} \, \|(h^{\perp})^{1/2} \Pi_{\mbf{V}} \bq \cdot \mbf{n} \|_{L^{\infty}(\Gamma_h)} L^2\, h \, \left( \|\erroru\|_{\compD} + \| I^{u} \|_{\compD}  \right) ^2, 
    \end{align*}
as well as
    \begin{align*}
	|(\kap^{-1}(u) \mbf{I}^{\bq}, \errorq )_{\mc{T}_h}| &\leq \dfrac{1}{2\, \delta_3} \| \kap^{-1/2}(u_h) \errorq \|^2_{\compD} + \dfrac{\delta_3}{2}\, \underline{\kappa}^{-2}\, \overline{\kap}   \| \mbf{I}^{\bq}\|^2_{\compD}, \\
	|(\kap^{-1}(u) - \kap^{-1}(u_h))\Pi_{\mbf{V}}\bq,\errorq)_{\mc{T}_h}| &\leq \dfrac{1}{2\, \delta_3} \,  \| \kap^{-1/2}(u_h) \errorq \|^2_{\compD} + \dfrac{\delta_3}{2}   \|\Pi_{\mbf{V}}\bq\|_{L^{\infty}(\compD)}^2 \,  \overline{\kap}\, L^2 \left( \|\erroru\|_{\compD} + \| I^{u} \|_{\compD}  \right)^2. \\
	|\md{T}^{f} | &\leq L_{f} \, (\|\erroru\|_{\compD} + \|I^u\|_{\compD})\, \|\erroru\|_{\compD},    
	\end{align*}
where $\delta_1,\delta_2,\delta_3$ are free positive parameters arising from applications of Young's inequality, $\overline{\kappa}$ and $\underline{\kappa}$ are the upper and lower bounds for the diffusivity, and $L$ and $L_f$ are the Lipschitz constants from $\kappa^{-1}$ and $f$ respectively.

If we let $\delta_1 = 4$, $\delta_2 = 12$, and $\delta_3 = 6$ in the above estimates and substitute back into \eqref{eq:norm_error_2} we obtain
	\begin{equation*}
	\triple{ (\errorq ,\erroru - \errortu, \varphi - \varphi_{h}) }_{{u_h}}^2 \leq 12\, \max \{ C_1\, h, C_2\} L^2  \, \left( \|\erroru\|_{\compD}^2 + \| I^{u} \|_{\compD}^2  \right)  + C\, L_{f} (\|\erroru\|_{\compD} + \|I^u\|_{\compD})\, \|\erroru\|_{\compD}  + C_3\,  \projerrorq^2 .
	\end{equation*}
where $C_1, C_2$ and $C_3$ only depend on $\overline{\kap},\underline{\kap},  R$, and the projections $\|(h^{\perp})^{1/2} \Pi_{\mbf{V}} \bq \cdot \mbf{n} \|_{L^{\infty}(\Gamma_h)}^2$ and  $\| \Pi_{\mbf{V}} \bq  \|_{L^{\infty}(\compD)}^2. $
\end{proof}
We now proceed to show that the approximation error in $u$ can be indeed controlled by the errors in the approximation of the flux, the hybrid variable and the transfer error, modulo the approximation properties of the discrete spaces. To show that, in the next lemma we will build upon the ideas as in \cite{SaSaSo2019} and use a duality argument.
\begin{lem}\label{lem:estimation E_u}
Assume that the Lipschitz constant is such that $L_{f}$ is small enough, and consider the discrete spaces to be of polynomial degree $k\geq 1$. Then, 
	\begin{align}\label{eq:estimation E_u}
    \|\erroru\|_{\compD}  \lesssim h^{1/2}\,  \triple{(\errorq,\erroru-\errortu, \varphi-\varphi_{h})}_{u_h} + (h^{1/2} + L\, h) \projerrorq + (L+ h^{1/2}) \projerroru.
	\end{align}
\end{lem}

\begin{proof}
The first part of the proof follows very closely the argument used in the proof of Lemma \ref{lem:StabilityS}. Given $\Theta \in L^2(\Omega)$ we will denote by $\boldsymbol\phi$ the solution to the dual problem \eqref{eq:DualProblem} associated to $\Theta$. Considering then the equations \eqref{eq:HDG_error-P1}, together with the dual system, it is possible to show that
	\begin{align}\label{eq:duality-descomp}
	(\varepsilon^u, \Theta)_{\mc{T}_h}& = \md{T}_{\bq}^1 + \md{T}_{\bq}^2 + \md{T}_{u} + \md{T}_{f},
	\end{align}
where
    \begin{alignat*}{10}
     \md{T}_{\bq}^1 &:= (\kap^{-1}(u_h) (\bq-\bq_h), \bsy{\Pi}_{\mbf{V}} \bsy{\phi} )_{\mc{T}_h} +   (\errorq ,\nabla \psi)_{\mc{T}_h}, \qquad&
     \md{T}_{\bq}^2 &:= (\kap^{-1}(u)-\kap^{-1}(u_h)) (\mbf{I}^{\bq} + \Pi_V\bq), \Pi_V \phi)_{\mc{T}_h}, \\
     \md{T}_{u} &:= \langle \errortu, P_M(\bsy{\phi}\cdot \mbf{n}) \rangle_{\Gamma_h} - \langle \errortq \cdot \mbf{n}, \Pi_W\psi\rangle_{\Gamma_h} \qquad&
     \md{T}_{f} &:= (f(u) - f(u_h), \Pi_W \psi)_{\mc{T}_h}.
    \end{alignat*}
To prove the result \eqref{eq:estimation E_u}, we will bound each of the terms $\md{T}_{\star}$, with $\star \in \{\bq, u, f\}$ in the decomposition \eqref{eq:duality-descomp}. 

\paragraph{Bound for $\md{T}_{f}^i$.} The simplest term to bound is $\md{T}_f$, for which an application of Cauchy-Schwartz, the properties of the HDG projector and the Lipschitz continuity of the source term $f$, together with the dual estimate \eqref{regularity dual problem} yield 
    \begin{equation}\label{eq:noundT_f}
    |\md{T}_f | \leq C_{\text{reg}}\, L_f \, (\|\erroru\|_{\compD} + \|I^u\|_{\compD} ) \|\Theta\|_{\Omega} .
    \end{equation}
Where the constant $C_{\text{reg}}$ is the stability constant from the dual problem \eqref{eq:DualProblem}.

\paragraph{Bound for $\md{T}_{\bq}^i$ .} Using the Lipschitz-continuity of $\kappa$ (c.f.  \eqref{eq:Lipschitz_k}) and following the arguments leading to equation (4.8) in \cite[Lemma 5]{SaSaSo2019}, the terms $\md{T}_{\bq}^1 $ and $ \md{T}_{\bq}^2 $ can be bounded like
    \begin{subequations}\label{bounds_Tq1-Tq2}
    \begin{align}
    |\md{T}_{\bq}^1 | &\leq \underline{\kap}^{-1/2} \,\|\kap^{-1/2}(u_h) (\errorq +  \mbf{I}^{\bq})\|_{\compD}\, \|\bsy{\Pi}_{\mbf{V}} \phi - \phi\|_{\compD}  + \| \mbf{I}^{\bq}\|_{\compD}\, \|\nabla(\psi-\psi_h) \|_{\compD} \nonumber \\
    &\leq  C\, C_{\text{reg}}\, \underline{\kap}^{-1/2}  h^{\min\{1,k\}} \|\kap^{-1/2}(u_h)\errorq\|_{\compD} \|\Theta\|_{\Omega} + 2 C\, C_{\text{reg}}\, \max\{\underline{\kap}^{-1/2}\,  \overline{\kap}^{-1/2}, 1 \} h^{\min\{1,k\}}\,  \| \mbf{I}^{\bq}\|_{\compD}\,\|\Theta\|_{\Omega} \\ 
    \intertext{and}
    |\md{T}_{\bq}^2 | &\leq C_{\text{reg}} ( \|\mbf{I}^{\bq}\|_{ \infty} + \| \bsy{\Pi}_{\mbf{V}} \bq\|_{\infty} ) \, L \, ( \|\erroru\|_{\compD} + \|I^u\|_{\compD} ) \|\theta\|_{\Omega} .
    \end{align}
    \end{subequations}
\paragraph{Bound for $\md{T}_{u}$.} To estimate this term we will have to decompose it and treat each of the parts separately. We will write then write $\md{T}_{u} := \sum_{i=1}^{11} \md{T}_{u}^i$, where:
    \begin{equation*}
    \begin{array}{lll}
	\md{T}_{u}^1 := -\langle \kap(u_h)\,  l^{-1}\, (\varphi(u) - \varphi(u_h)), \psi + l \partial_n \psi \rangle_{\Gamma_h}, &  & \md{T}_{u}^7 := -\langle\tau(\erroru - \errortu), P_M\psi \rangle _{\Gamma_h}, \\
	\md{T}_{u}^2 := -\langle \kap(u_h)\, (\varphi(u) - \varphi_{h}(u_h)), (Id_M - P_M)\partial_n \psi \rangle_{\Gamma_h}, &  &  \md{T}_{u}^8 := -\langle \kap(u_h) \, (\varphi(u) - \varphi_{h}(u_h))\, \delta_{\mbf{I}^{\bq}}, \psi  \rangle_{\Gamma_h} \\
	\md{T}_{u}^3 := \langle \delta_{\mbf{I}^q}, \psi\rangle_{\Gamma_h}, &  & \md{T}_{u}^9 := -\langle \kap(u_h) \, (\varphi(u) - \varphi_{h}(u_h))\, \delta_{\bsy{\Pi}_{\mbf{V}} \bq}, \psi  \rangle_{\Gamma_h} , \\
	\md{T}_{u}^4 := \langle \mbf{I}^{\bq} \cdot \mbf{n}, (Id_M-P_M) \psi \rangle_{\Gamma_h} , &  & \md{T}_{u}^{10} := -\langle \kap(u_h) \, (\varphi(u) - \varphi_{h}(u_h))\, \mbf{I}^{\bq}\cdot \mbf{n}, \psi  \rangle_{\Gamma_h}. \\
	\md{T}_{u}^5 := -\langle \tau P_M I^u, \psi\rangle_{\Gamma_h}, &  & \md{T}_{u}^{11} := -\langle \kap(u_h) \, (\varphi(u) - \varphi_{h}(u_h))\, \bsy{\Pi}_{\mbf{V}} \bq \cdot \mbf{n}, \psi  \rangle_{\Gamma_h}, \\
	\md{T}_{u}^6 := \langle \delta_{\errorq}, \psi\rangle_{\Gamma_h}. &  & 
	\end{array}
	\end{equation*}
\paragraph{Bounds for $\md{T}_u^1 - \md{T}_u^7$:} These terms can be estimated by we applying the same techniques of \cite[Lemma 6]{SaSaSo2019}. We will omit most of the the details here. Recalling that the length of the transfer path $l(\boldsymbol{x})\leq c\, R \, h \quad  \forall \boldsymbol{x} \in \Gamma_h$  and considering the constant $\tilde{c}$ from Lemma \ref{Lema auxiliar para T_u} \cite[Lemma 6]{SaSaSo2019}, we have
 \begin{equation*}
    \begin{array}{rl}
	|\md{T}_{u}^1| &\leq c\, \tilde{c}\, \overline{\kap}^{1/2}\,  R\,  h\, \| \kap^{1/2}(u_h)\, l^{-1/2} \, ( \varphi(u) - \varphi_{h}(u_h)) \|_{\Gamma_h} \|\Theta\|_{\Omega}, \\
	|\md{T}_{u}^2| &\leq c\, \tilde{c}\, \overline{\kap}^{1/2} \, R^{1/2} \, h \,  \| \kap^{1/2}(u_h)\, l^{-1/2} \, ( \varphi(u) - \varphi_{h}(u_h)) \|_{\Gamma_h} \|\Theta\|_{\Omega}, \\
	|\md{T}_{u}^3| &\leq \frac{1}{\sqrt{3}}\, c^{1/2}\, \tilde{c} \,   R^{3/2}\, h^{1/2} \|h^{\perp} \partial_n \mbf{I}^{\bq} \cdot \mbf{n} \|_{\Omega_h^c} \|\Theta\|_{\Omega}, \\
	|\md{T}_{u}^4| &\leq \tilde{c}\, h \|(h^{\perp})^{1/2} \mbf{I}^{\bq} \cdot \mbf{n} \|_{\Gamma_h} \|\Theta\|_{\Omega}, \\
	|\md{T}_{u}^5| &\leq \tilde{c}\, \overline{\tau} \,R \, h^{1/2} \|(h^{\perp})^{1/2} I^u \|_{\Gamma_h} \|\Theta\|_{\Omega}, \\
	|\md{T}_{u}^6| &\leq \frac{1}{\sqrt{3}} \, c^{1/2}\, \tilde{c} \,   \overline{\kappa}^{1/2} \, \max_{e \in \mc{E}_h^{\partial}}\{C^e_{\text{ext}}, C^e_{\text{inv}}\} R^2\, h^{1/2} \, \| \kap^{-1/2}(u_h)\,  \errorq  \|_{\Omega_h} \|\Theta\|_{\Omega},\\
	|\md{T}_{u}^7| &\leq c\, \tilde{c}\, \overline{\tau}^{1/2} \, R\, h \|\tau^{1/2}\, (\erroru-\errortu) \|_{\partial \mc{T}_h} \|\Theta\|_{\Omega}.
\end{array}
\end{equation*}
\paragraph{Bounds for $\md{T}_u^8 - \md{T}_u^9$:} Let us first notice that by definition of $\md{T}_u^8$, we can obtain
    \begin{equation*}
    |\md{T}_u^8| = \left|  \left\langle \kap^{1/2}(u_h)\, l \, \kap^{1/2}(u_h) \, l^{-1/2} (\varphi(u) - \varphi_h(u_h))\, l^{1/2}\, \delta_{\mbf{I}^{\bq}}, l^{-1} \, \psi \right\rangle_{\Gamma_h} \right| .   
    \end{equation*}
Then, by Cauchy-Schwartz, the fact that $l(\boldsymbol{x})\leq c\, R \, h \quad  \forall \boldsymbol{x} \in \Gamma_h$ and the boundedness of $\kappa$, we can obtain 
    \begin{equation*}
    |\md{T}_u^8| \leq c\, \overline{\kap}^{1/2}\,  R\, h \, \| \kap^{1/2}(u_h)\, l^{-1/2} \, ( \varphi(u) - \varphi_{h}(u_h)) \, l^{1/2}\, \delta_{\mbf{I}^{\bq}} \|_{\Gamma_h} \|  l^{-1} \, \psi\|_{\Gamma_h}.
    \end{equation*}
Finally, a direct application of \eqref{ineq: est delta 3} to the factor involving the function $\delta_{\boldsymbol I^q}$, and using the estimation \eqref{estim-aux for Tu-p4} for the factor $\|  l^{-1} \, \psi\|_{\Gamma_h}$, results in
    \begin{equation*}
    \begin{array}{rl}
	|\md{T}_{u}^8| &\leq \frac{1}{\sqrt{3}}\, c \, \tilde{c}\, \overline{\kap}^{1/2} \, R^2\, h\,   \displaystyle \sup_{\mbf{x} \in \Gamma_h} \|(h^{\perp}\, \partial_n  \mbf{I}^{\bq} \cdot \mbf{n}\|_{l(\mbf{x})}\,  \| \kap^{1/2}(u_h)\, l^{-1/2} \, ( \varphi(u) - \varphi_{h}(u_h)) \|_{\Gamma_h} \, \|\Theta\|_{\Omega}.
    \end{array}
    \end{equation*}
Analogously, we can show that
    \begin{equation*}
    \begin{array}{rl}
	|\md{T}_{u}^9| &\leq \frac{1}{\sqrt{3}}\, c \, \tilde{c}\, \overline{\kap}^{1/2} \, R^2\, h \, \displaystyle \sup_{\mbf{x} \in \Gamma_h} \|(h^{\perp}\, \partial_n  \mbf{\Pi}_{\mbf{v}} \bq \cdot \mbf{n}\|_{l(\mbf{x})}\,   \| \kap^{1/2}(u_h)\, l^{-1/2} \, ( \varphi(u) - \varphi_{h}(u_h)) \|_{\Gamma_h} \,  \|\Theta\|_{\Omega}.
    \end{array}
    \end{equation*}

\paragraph{Bounds for $\md{T}_u^{10} - \md{T}_u^{11}$:}  We start as in the case for $\md{T}_u^{10} - \md{T}_u^{11}$ by combining, Cauchy-Schwartz , the bounds for $\kappa$ and $l\,\lesssim Rh$, to obtain
    \begin{equation*} 
    \begin{array}{rl}
    |\md{T}_u^{10}| &=  \left|  \left\langle \kap^{1/2}(u_h)\, l \, \kap^{1/2}(u_h) \, l^{-1/2} (\varphi(u) - \varphi_h(u_h))\, l^{1/2}\, \mbf{I}^{\bq} \cdot \mbf{n} , l^{-1} \, \psi \right\rangle_{\Gamma_h} \right|   \\
    &\leq  c\, \overline{\kap}^{1/2} \, R\, h\,  \| \kap^{1/2}(u_h) \, l^{-1/2} (\varphi(u) - \varphi_h(u_h))\, l^{1/2} \mbf{I}^{\bq} \cdot \mbf{n} \|_{\Gamma_h} \| l^{-1} \, \psi\|_{\Gamma_h}. \\
    &\leq  c\, \overline{\kap}^{1/2} \, R\, h\,  \| \kap^{1/2}(u_h) \, l^{-1/2} (\varphi(u) - \varphi_h(u_h))  \|_{\Gamma_h } \| l^{1/2} \mbf{I}^{\bq} \cdot \mbf{n} \|_{L^{\infty} (\Gamma_h}) \| l^{-1} \, \psi\|_{\Gamma_h}.
    \end{array}
    \end{equation*}
From here, we will use the inequality $l(\boldsymbol{x}) \leq r_e h_e^{\perp}$ together with the estimate \eqref{estim-aux for Tu-p4} for $\| l^{-1}\, \psi\|_{\Gamma_h}$, we get
    \begin{equation*}
    \begin{array}{rl}
	|\md{T}_{u}^{10}| &\leq c \, \tilde{c} \,   \overline{\kap}^{1/2} \, R^{3/2}\, h\, \| \kap^{1/2}(u_h)\, l^{-1/2} \, ( \varphi(u) - \varphi_{h}(u_h)) \|_{\Gamma_h} \, \|(h^{\perp})^{1/2}\, \mbf{I}^{\bq}\cdot \mbf{n}\|_{L^{\infty}(\Gamma_h)} \|\Theta\|_{\Omega}.
    \end{array}
    \end{equation*}
Similar arguments can be used to derive the following analogous bound
    \begin{equation*}
    \begin{array}{rl}
    |\md{T}_{u}^{11}| &\leq  c \, \tilde{c} \,   \overline{\kap}^{1/2} \, R^{3/2}\, h\, \| \kap^{1/2}(u_h)\, l^{-1/2} \, ( \varphi(u) - \varphi_{h}(u_h)) \|_{\Gamma_h} \, \|(h^{\perp})^{1/2}\, \bsy{\Pi}_{\mbf{V}} \bq\cdot \mbf{n}\|_{L^{\infty}(\Gamma_h)} \|\Theta\|_{\Omega}.
    \end{array}
    \end{equation*}
Finally, letting $\Theta=\erroru$ in $\compD$ and $\Theta=0$ in $\compD^c$ and using the estimates derived above for all the terms $\mathbb T_{\star}$ in the decomposition \eqref{eq:duality-descomp}, one arrives at the desired estimate:
    \begin{equation*}
    \|\erroru\|_{\compD}  \lesssim h^{1/2}\,  \triple{(\errorq,\erroru-\errortu, \varphi-\varphi_h)}_{u_h} + (h^{1/2} + L\, h) \projerrorq + (L+ h^{1/2}) \projerroru,
    \end{equation*}
\end{proof}
This concludes the analysis of the discretization for problems with nonlinear diffusivities of the form $\kappa =\kappa(u)$. The reminder of the article will be devoted to the analysis of cases where the nonlinearity appears as a dependence to the gradient of the unknown. This functional dependence will require a different reformulation of the problem.
%
% =====================================================
\section{Non-linearities of the form $\kappa(\nabla u)$}\label{sec:DependenceOnGradU}
% =====================================================
%
% ==============================================
\subsection{Problem statement}
% ==============================================
%
In some applications, the diffusivity coefficient depends on the gradient of the solution, rather than on the solution itself. This is indeed the case, for instance, in the plasma equilibrium problem, where the coefficient is the inverse of the magnetic permeability. In ferromagnetic materials, the magnetic permeability becomes a function of the total magnetic field and therefore the coefficient has the functional dependence $\kappa = \kappa(\nabla u)$. In cases like this, we will be interested in boundary value problems of the form
	\begin{subequations}\label{eq:main2}
	\begin{align}
	&& -\nabla \cdot \left( \kap(\nabla u) \nabla u \right) &= f(u) & &\text{ in } \Omega, && \\
	&& u &= g  & &\text{ on } \Gamma:=\partial \Omega. &&
	\end{align}
	\end{subequations}
where, just like in the previous section, the source term $f$ will be assumed to be Lipschitz-continuous in $\Omega$, with Lipschitz constant $L_f>0$. We will also maintain the assumption \eqref{eq:assumptions_kappa} on boundedness of the permeability. Note that, since we will be searching for solutions with $H^1(\Omega)$ regularity, the hypothesis \eqref{eq:assumptions_kappa} will guarantee that the permeability remains bounded as a function of $\nabla u$. The Lipschitz-continuity assumptions \eqref{eq:Lipschitz_k}, \eqref{eq:Lipschitz_k1/2}, and \eqref{eq:Lipschitz_k-1} will be replaced by their following vector counterparts 
	\begin{equation}\label{eq:Lipschitz_k-P2}
	\begin{array}{rclcl}
    \|\kap^{-1}(\bsigma_1) - \kap^{-1}(\bsigma_2) \|_{L^2(\Gamma)} &\leq&  L \|\bsigma_1 - \bsigma_2\|_{\boldsymbol L^2(\Omega)} &\qquad& \forall \, \bsigma_1, \bsigma_2 \in \boldsymbol L^2(\Omega), \\[2ex]
	\|\kap^{1/2}(\bsigma_1) - \kap^{1/2}(\bsigma_2) \|_{L^{\infty}(\Gamma)} &\leq& \what{L} \|\bsigma_1 - \bsigma_2\|_{\boldsymbol L^2(\Omega)} &\qquad& \forall \, \bsigma_1, \bsigma_2 \in \boldsymbol L^2(\Omega), \\[2ex]
    \|\kap(\bsigma_1) - \kap(\bsigma_2) \|_{L^2(\Omega)} &\leq& \widetilde L \|\bsigma_1 - \bsigma_2\|_{\boldsymbol L^2(\Omega)} &\qquad& \forall \, \bsigma_1, \bsigma_2 \in \boldsymbol L^2(\Omega).
	\end{array}
	\end{equation}
Following the spirit of reformulating the problem in a mixed form, the functional dependence $\kappa(\nabla u)$ will require us to introduce a new auxiliary variable. Therefore, we introduce $\bsigma:= \nabla u$ and will express the the flux as $\bq:= -\kap(\bsigma) \bsigma$, thus introducing two additional unknowns to the problem. With these definitions, it is possible to write \eqref{eq:main2} as the equivalent system 
	\begin{align*}
	&& && \bsigma - \nabla u &= 0& &\text{ in } \Omega, \\
	&& && \bq + \kap(\bsigma)\, \bsigma &= 0 & &\text{ in } \Omega, \\
	&& && \nabla \cdot \bq  &= f(u) & &\text{ in }  \Omega, \\
	&& && u &= g &  &\text{ on } \partial \Omega.
	\end{align*}
We shall analyze the discretization of this system when restricted to the subdomain in $\compD$. In view of this, our target formulation is
	\begin{subequations}\label{eq:mixed numerical domain}
	\begin{align}
	&& && \bsigma - \nabla u &= 0& &\text{ in } \compD, &&\label{eq:mixed numerical domain_a}\\
	&& && \bq + \kap(\bsigma)\, \bsigma &= 0 & &\text{ in } \compD, &&\label{eq:mixed numerical domain_b}\\
	&& && \nabla \cdot \bq  &= f(u) & &\text{ in }  \compD, &&\label{eq:mixed numerical domain_c}\\
	&& && u &= \varphi &  &\text{ on } \Gamma_h = \partial \compD .&&\label{eq:mixed numerical domain_d}
	\end{align}
	\end{subequations}
where the boundary conditions have been transferred by means of 
		\begin{equation*}
	\varphi(\boldsymbol \sigma,\mbf{x}):= g(\overline{\mbf{x}}) + \int_0^{l(\mbf{x})} \big( \kap^{-1}(\bsigma)  \mbf{q} \big) (\mbf{x} + \boldsymbol{t}(\mbf{x})s) \cdot  \boldsymbol{t}(\mbf{x}) ds.
	\end{equation*}
The HDG scheme associated to \eqref{eq:mixed numerical domain} reads: Find $(\bq_h, \bsigma_h, u_h, \what{u}_h) \in \mbf{V}_h\times \mbf{V}_h \times  W_h \times M_h$, such that
	\begin{subequations}\label{eq:HDG-P12}
	\begin{align}
	(\bsigma_h, \mbf{v})_{\mc{T}_h} + (u_h, \nabla \cdot \mbf{v})_{\mc{T}_h} - \pdual{\what{u}_h, \mbf{v} \cdot \mbf{n}}_{\partial \mc{T}_h} &=  0,  \\
	( \bq_h, \mbf{s} )_{\mc{T}_h} + ( \kap(\bsigma_h) \, \bsigma_h, \mbf{s})_{\mc{T}_h} &= 0 , \\
	-(\bq_h, \nabla w)_{\mc{T}_h} + \pdual{ \what{\bq}_h \cdot \mbf{n},w}_{\partial \mc{T}_h} &=  (f(u_h),w)_{\mc{T}_h}, \label{eq:HDG-P12_b} \\
	\langle \what{u}_h,\mu\rangle_{\Gamma_h} &=  \langle \varphi_{h}(\bsigma_h), \mu\rangle_{\Gamma_h},  \\
	\langle\what{\bq}_h \cdot \mbf{n}, \mu\rangle_{\partial \mc{T}_h \setminus \Gamma_h} &= 0,
	\end{align}
for all $(\mbf{v}, \mbf{s}, w, \mu)\in \mbf{V}_h\times \mbf{V}_h \times W_h \times M_h$. Here, the spaces $\boldsymbol V_h$, $W_h$, and $M_h$ have been defined in \eqref{eq:PolynomialSpaces}, the restriction to the mesh skeleton of the numerical flux has been defined as 
\[
    \what{\bq}_h\cdot \mbf{n} := \bq_h \cdot \mbf{n} +  \tau(u_h - \what{u}_h) \quad \textrm{on}\quad \partial \mc{T}_h,
\]
and the approximate boundary condition given by
		\begin{equation}
	\varphi_{h}(\bsigma_h, \mbf{x}):= g(\overline{\mbf{x}}) + \int_0^{l(\mbf{x})} \big( \kap^{-1}(\bsigma_h)  \mbf{q}_h \big) (\mbf{x} + \boldsymbol{t}(\mbf{x})s) \cdot  \boldsymbol{t}(\mbf{x}) ds.
	\end{equation}
	\end{subequations}
As before, the maximum value of the positive stabilization function $\tau$ will be denoted by $\overline{\tau}$.

In this section we will analyze an HDG scheme for problems of this form. We will first have to reformulate the problem in terms of a mixed system with one additional unknown when compared to the case analyzed in the previous section. Many of the arguments required for the analysis will be similar to those developed in the previous section, and the analysis technique is similar as well. We will therefore omit many of the technical details and indicate the main steps in the analysis, focusing on those that are different from the previous section.
%
% ===========================
\subsection{Well-posedness}\label{sec:well-posedness2}
% ==========================
%
The proof that the system \eqref{eq:HDG-P1} is well-posed will rely on a fixed point argument. As in the previous section, we define the operator $\mc{J}_2 : \mbf{V}_h\times W_h  \to \mbf{V}_h\times W_h $ that maps a pair of functions $(\boldsymbol\eta,\zeta)$ to the first and third component of the solution $(\bq, \bsigma, u, \what{u})\in \mbf{V}_h\times\mbf{V}_h\times W_h \times M_h$ to the HDG system \eqref{eq:HDG-P12} where the source has been evaluated at $(\boldsymbol\eta, \zeta)$. Namely
	\begin{subequations}\label{eq:Fixed_point2}
	\begin{align}
	(\bsigma, \mbf{v})_{\mc{T}_h} + (u, \nabla \cdot \mbf{v})_{\mc{T}_h} - \pdual{\what{u}, \mbf{v} \cdot \mbf{n}}_{\partial \mc{T}_h} &=  0,  \\
	( \bq, \mbf{s} )_{\mc{T}_h} + ( \kap(\boldsymbol\eta) \, \bsigma, \mbf{s})_{\mc{T}_h} &= 0 , \label{eq:Fp-b} \\
	-(\bq, \nabla w)_{\mc{T}_h} + \pdual{ \what{\bq} \cdot \mbf{n},w}_{\partial \mc{T}_h} &=  (f(\zeta),w)_{\mc{T}_h},  \\
	\langle \what{u},\mu\rangle_{\Gamma_h} &=  \langle \varphi(\boldsymbol\eta), \mu\rangle_{\Gamma_h},  \\
	\langle\what{\bq} \cdot \mbf{n}, \mu\rangle_{\partial \mc{T}_h \setminus \Gamma_h} &= 0,
	\end{align}
	\end{subequations}
for all $(\mbf{v},\mbf{s}, w, \mu)\in \mbf{V}_h \times\mbf{V}_h \times W_h \times M_h$. Just as before, $\varphi(\zeta)$ accounts for the transferred boundary conditions and, since the discrete linearized system above is uniquely solvable \cite{CoQiuSo2014}, $\mathcal J_2$ is well defined.

Given a function $\boldsymbol\eta\in \boldsymbol V_h$, we define the following norm over the product space $\mbf{V}_h \times\mbf{V}_h \times W_h \times M_h$
    \begin{align}\label{def:triple_norm-P2}
          \triple{(\mbf{s}, \mbf{v},w,\lambda)}_{\boldsymbol\eta} &:= \left( \| \mbf{s}\|^2_{\compD} +   \| \kap^{1/2}(\boldsymbol\eta) \mbf{v} \|_{\compD}^2 + \|\tau^{1/2} \, w\|_{\partial \mc{T}_h}^2 + \| \kap^{1/2}(\boldsymbol\eta)\,  l^{-1/2} \, \lambda(\boldsymbol\eta) \|_{\Gamma_h}^2 \right)^{1/2}.
    \end{align}	
The general road map for the proof is as follows. Lemmas \ref{lem:EstimateNormH2}  and \ref{lem:StabilityS2} below, will allow us to control $(\bsy{q}, \bsigma, u-\what{u}, \varphi)$ by the linearized source term $f(\zeta)$ and the boundary condition \textit{at the physical boundary}, $\overline{g}$. An application of these two results will then allow us---modulo some technical assumptions involving the bound of the diffusivity and the distance between the physical and computational domains---to use the Lipschitz continuity of $f$ and $\kappa$ to prove that the mapping is indeed a contraction. This will be done in Theorem \ref{thm:Fixed_Point-2}, from which the well posedness of the HDG system \eqref{eq:HDG-P12} will follow as a simple corollary.
\begin{lem}\label{lem:EstimateNormH2} 
If Assumptions \eqref{eq:Assumptions} hold, then
	\begin{equation}
	\label{eq:lemEstimateNormH2}
    \triple{(\bsy{q}, \bsigma, u-\what{u}, \varphi)}_{\boldsymbol\eta}^2 \leq  \max\{1, \overline{\kap}\} \Big(  4\|f(\zeta)\|_{\compD} \|u\|_{\compD}  + 6 \|\kap^{1/2}(\boldsymbol\eta)\, l^{-1/2} \, \g \|_{\Gamma_h}^2 \Big).
    \end{equation}
\end{lem}
\begin{proof}
Note that if we let $\mbf{s}=\bsy{q}$ in \eqref{eq:Fp-b} , we have
    \begin{equation}\label{eq:Fp1}
    \|\bsy{q}\|_{\compD}^2 \leq \overline{\kap}\,  \|\kap^{1/2}(\boldsymbol\eta)\,\bsigma\|_{\compD}^2. 
    \end{equation}
Then, following the process outlined in the proof of Lemma \ref{lem:EstimateNormH}, we go back to \eqref{eq:Fixed_point2} and choose the test functions as $\mbf{v} = -\mbf{q}, \mbf{s}=\bsigma, w=u$, and \[
\mu=\left\{ 
    \begin{array}{cc}
    - \what{u}, &  \text{ on }\partial \mc{T}_h \setminus \Gamma_h\\
    -\what{\bsy{q}}\cdot \mbf{n}, & \text{ on } \Gamma_h.
    \end{array}
\right.
\]
This leads to the equality
\begin{align*}
     \|\kap^{1/2}(\boldsymbol\eta) \bsigma\|^2_{\compD} + \| \tau^{1/2}(u-\what{u}) \|^2_{\partial \mc{T}_h} =\,& \;\; (f(\zeta),u)_{\mc{T}_h} - \langle \varphi(\boldsymbol\eta), \kap(\boldsymbol\eta)l^{-1} \varphi(\boldsymbol\eta) \rangle_{\Gamma_h} +\langle\varphi(\boldsymbol\eta), \kap(\bsigma)l^{-1} \g  \rangle_{\Gamma_h}   \\
         & + \langle \varphi(\boldsymbol\eta), \g \rangle_{\Gamma_h}  - \langle\varphi(\boldsymbol\eta), \tau(u-\what{u}) \rangle_{\Gamma_h} .
\end{align*}
The terms on the right hand side can be estimated by an application of Lemma \ref{lem:EstimateVarphi}, yielding
\begin{align*}
 \| \kap^{1/2}(\boldsymbol\eta) \bsigma \|_{\compD}^2 + \|\tau^{1/2} \, (u-\what{u})\|_{\partial \mc{T}_h}^2 + \| \kap^{1/2}(\eta)\,  l^{-1/2} \, \varphi(\boldsymbol\eta) \|_{\Gamma_h}^2 &\leq2\|f(\zeta)\|_{\compD} \|u\|_{\compD}  + 3 \|\kap^{1/2}(\boldsymbol\eta)\, l^{-1/2} \, \g \|_{\Gamma_h}^2.
	\end{align*}
Combining this estimate with \eqref{eq:Fp1}, we obtain 
	\begin{align*}
\triple{(\bsy{q}, \bsigma, u-\what{u}, \varphi)}_{\boldsymbol\eta}^2 &\leq  \max\{1, \overline{\kap}\} \left(  4\|f(\zeta)\|_{\compD} \|u\|_{\compD}  + 6 \|\kap^{1/2}(\boldsymbol\eta)\, l^{-1/2} \, \g \|_{\Gamma_h}^2 \right),
	\end{align*}
which concludes the proof.
\end{proof}
It only remains now to estimate the norm of $u$ in terms of the sources and the boundary conditions. This will be done in the next lemma. 
\begin{lem}\label{lem:StabilityS2}
Suppose that Assumptions \eqref{eq:Assumptions} and \eqref{regularity dual problem} are satisfied. Then, there exists $\what{c}>0$, independent of $h$ such that 
	\begin{equation}\label{estimate-u-2problem}
	\|u\|_{\compD} \leq 4\, \max\{ \what{c}^2h,1 \}\,  \|f(\zeta)\|_{\compD} +  2\, (\sqrt{3}\, \what{c} + \overline{\kap}^{1/2}\, R^{1/2} )\,h^{1/2}\, \|\kap^{1/2}(\boldsymbol\eta) \,  l^{-1/2} \,  \g \| _{\Gamma_h}.
	\end{equation}
\end{lem}
\begin{proof}
The proof of this result is follows, with small variations, the same process as that of Lemma \ref{lem:StabilityS}. By using $\boldsymbol\eta$  instead of $\zeta$ in the dual problem given in \eqref{eq:DualProblem}, and splitting the duality product as
	\begin{equation*}
	(u, \Theta)_{\mc{T}_h} = \md{T}_{\bsigma} + \md{T}_{u} + \md{T}_{f},
	\end{equation*}	 
where 
    \begin{equation*}
    \md{T}_{\bsigma} := - (\bsigma, \bsy{\Pi}_{\mbf{V}} \bsy{\phi} - \bsy{\phi})_{\mc{T}_h}, \quad 
    \md{T}_{u} := \langle \what{u}, P_M(\bsy{\phi} \cdot \mbf{n}) \rangle_{\Gamma_h} - \langle\what{\bq} \cdot \mbf{n} , \Pi_W \psi\rangle_{\Gamma_h} \quad \text{and} \quad \md{T}_{f} := (f(\zeta), \Pi_W \psi)_{\mc{T}_h}. 
    \end{equation*}
The terms $\md{T}_{\bsigma}$ and $\md{T}_{f}$ are bounded as 
    \begin{align*}
    |\md{T}_{\bsigma}| &\lesssim \underline{\kap}^{-1/2}\, h \|\kap^{-1/2}(\boldsymbol\eta)\, \bsigma\|_{\compD} \, \|\Theta\|_{\Omega}, \quad \text{and} \quad |\md{T}_{f} | \lesssim  \|f(\zeta)\|_{\compD}\, \|\Theta\|_{\Omega},
    \end{align*}
and, we rewrite $\md{T}_{u}$ as $\md{T}_{u} = \sum_{i=1}^5 \md{T}_{u}^i$, with
    \begin{equation*}
    \begin{array}{lcl}
     \md{T}_{u}^1 := - \langle\kap(\boldsymbol\eta) l^{-1} \varphi(\boldsymbol\eta) ,\psi + l \partial_n \psi\rangle_{\Gamma_h}, & & \md{T}_{u}^4 := - \langle\tau(u - \what{u}) ,P_M \psi\rangle_{\Gamma_h}, \\[2ex]
    \md{T}_{u}^2 := \langle\kap(\boldsymbol\eta) \varphi(\boldsymbol\eta), (P_M - I_d)\partial_n \psi \rangle_{\Gamma_h}, &&
	\md{T}_{u}^5 := \langle\kap(\boldsymbol\eta)\,  l^{-1} \, \g, \psi\rangle_{\Gamma_h}, \\[2ex]
	\md{T}_{u}^3 := \langle \delta_{\bsigma}, \psi \rangle_{\Gamma_h} .&& 
	\end{array}
    \end{equation*}
These terms can be bounded using arguments analogous to those in Lemma \ref{lem:StabilityS}, yielding the desired estimate \eqref{estimate-u-2problem}.
\end{proof}

The results in the two preceding lemmas can be combined to estimate $(\bsy{q}, \bsigma, u-\what{u}, \varphi)$ in terms of the source $f(\zeta)$ and the boundary data $\g$. This follows readily from an application of Lemma \ref{lem:StabilityS2} to \eqref{eq:lemEstimateNormH2}, yielding
    \begin{align}
    \nonumber
    \triple{(\bsy{q}, \bsigma, u-\what{u}, \varphi)}^2_{\eta} \leq\,&\;\;\; \left( 16 \, \max\{ 1,\overline{\kap}\}^2 + 8 \,  \max\{ \what{c}^2h,1 \}^2 \right) \|f(\zeta)\|_{\compD}^2 \\[2ex]
    \label{estimate_triple-2}
    & + \left( 6 \max\{1,\overline{\kap}\} + 2\, (\sqrt{3}\, \what{c} + \overline{\kap}^{1/2}\, R^{1/2} )^2\,h \right) \|\kap^{1/2}(\boldsymbol\eta) \,  l^{-1/2} \,  \g \|^2_{\Gamma_h}.
    \end{align}
In turn,  \eqref{estimate-u-2problem} implies that
    \begin{equation}\label{estiamte_norm-u-2}
    \begin{array}{rl}
     \|u\|_{\compD}^2  &\leq 32\,  \max\{ \what{c}^2h,1 \}^2\,  \|f(\zeta)\|_{\compD}^2 + 8\, (\sqrt{3}\, \what{c} + \overline{\kap}^{1/2}\, R^{1/2} )^2 \,h\, \|\kap^{1/2}(\boldsymbol\eta) \,  l^{-1/2} \,  \g \| _{\Gamma_h}^2. 
    \end{array}
    \end{equation}
These two estimates will be used to prove the contractive properties of $\mathcal J_2$ as we will now show.

 \begin{thm}\label{thm:Fixed_Point-2}
Suppose that the dual regularity \eqref{regularity dual problem} and the Assumptions \eqref{eq:Assumptions} hold and suppose also that 
    \begin{equation}\label{assumption1-Fixed-Point-P2}
    \bigg( 16 \, \max\{ 1,\overline{\kap}\}^2 + 40 \,  \max\{ \what{c}^2h,1 \}^2 \bigg) L^2_{f}  < \dfrac{1}{4}  \end{equation}
and
    \begin{equation}\label{assumption2-Fixed-Point-P2}
    \bigg( 6 \max\{1,\overline{\kap}\} + 10\, (\sqrt{3}\, \what{c} + \overline{\kap}^{1/2}\, R^{1/2} )^2\,h \bigg) \what{L}^2  \|l^{-1/2} \,  \g \|^2_{\Gamma_h} < \dfrac{1}{4} .
    \end{equation}
are satisfied. Then $\mc{J}_2$ is a contraction operator.
\end{thm}

\begin{proof}
Let $(\boldsymbol\eta_i,\zeta_i) \in \mbf{V}_h \times W_h $ and define $(\bsigma_i,u_i):= \mathcal J_2\left((\boldsymbol\eta_i,\zeta_i)\right) \in \mbf{V}_h \times W_h$ for $i \in \{1,2\}$. Then,
    \begin{equation*}
    \|\mc{J}_2(\boldsymbol\eta_1,\zeta_1) - \mc{J}_2(\boldsymbol\eta_2,\zeta_2)\|_{\compD} = \|(\bsigma_1-\bsigma_2, u_1-u_2)\|_{\compD} = \left(  \|\bsigma_1 - \bsigma_2\|^2_{\compD} + \|u_1 - u_2\|^2_{\compD} \right)^{1/2}.
    \end{equation*}
By applying the inequalities \eqref{estimate_triple-2} and \eqref{estiamte_norm-u-2} respectively to  $(\bsigma_1-\bsigma_2)$ and $(u_1-u_2)$, we obtain
    \begin{align*}
    \|\mc{J}_2(\boldsymbol\eta_1,\zeta_1) -\,& \mc{J}_2(\boldsymbol\eta_2,\zeta_2)\|_{\compD} \leq \bigg( \left( 16 \, \max\{ 1,\overline{\kap}\}^2 + 40 \,  \max\{ \what{c}^2h,1 \}^2 \right) \|f(\zeta_1) - f(\zeta_2)\|_{\compD}^2 \\
    &+ \left( 6 \max\{1,\overline{\kap}\} + 10\, (\sqrt{3}\, \what{c} + \overline{\kap}^{1/2}\, R^{1/2} )^2\,h \right) \|\kap^{1/2}(\boldsymbol\eta_1) - \kap^{1/2}(\boldsymbol\eta_2) \|_{L^{\infty}(\Gamma_h)}^2 \|l^{-1/2} \,  \g \|^2_{\Gamma_h} \bigg)^{1/2}.
    \end{align*}
Then, using the Lipschitz continuity of $f$ and $\kappa^{1/2}$, we get
    \begin{align*}
    \nonumber
    \|\mc{J}_2(\boldsymbol\eta_1,\zeta_1) -\,& \mc{J}_2(\boldsymbol\eta_2,\zeta_2)\|_{\compD} \leq \bigg( \left( 16 \, \max\{ 1,\overline{\kap}\}^2 + 40 \,  \max\{ \what{c}^2h,1 \}^2 \right) L^2_{f} \|\zeta_1 - \zeta_2\|_{\compD}^2 \\
    \label{estimate1_FP}
    & + \left( 6 \max\{1,\overline{\kap}\} + 10\, (\sqrt{3}\, \what{c} + \overline{\kap}^{1/2}\, R^{1/2} )^2\,h \right) \what{L}^2 \|\boldsymbol\eta_1 - \boldsymbol\eta_2\|_{L^{\infty}(\compD)}^2 \|l^{-1/2} \,  \g \|^2_{\Gamma_h} \bigg)^{1/2}.
    \end{align*}
The proof is concluded, by applying the assumptions \eqref{assumption1-Fixed-Point-P2} and \eqref{assumption2-Fixed-Point-P2} to the right hand side of the preceding inequality.
\end{proof}
As a result we can then conclude this section with with the following
\begin{crl}\label{crl:well-posed}
If the hypotheses of Theorem \ref{thm:Fixed_Point-2} are satisfied, the HDG system \eqref{eq:Fixed_point2} is well posed.
\end{crl}
Having established the well posedness of the discrete system \eqref{eq:Fixed_point2}, we will concentrate our efforts in the next section on establishing the convergence properties of the HDG scheme.
%
% ==============================================
\subsection{\textit{A priori} error analysis}\label{sec:Prob2Apriori}
% ==============================================
%
The study of the convergence properties and rates of our discretization follows similar steps as the ones laid out in Section \ref{sec:Prob1Apriori}, adapted for the extended system that arises from the introduction of the additional auxiliary variable $\boldsymbol \sigma = \nabla u$. To avoid unnecessary repetition of arguments, we will focus on the differences between these two cases and will omit most of the details that can be easily inferred from Section \ref{sec:Prob1Apriori}.

As before, we decompose the error with the aid of the HDG projection as :
\[
\bq-\bq_h = \bsy{\varepsilon}^{\bq}+\mbf{I}^{\bq},\quad  \bsigma-\bsigma_h =  \bsy{\varepsilon}^{\bsigma}+\mbf{I}^{\bsigma}, \;\;\text{ and }\;\; u-u_h =  \varepsilon^u + I^u,
\]
where, similar to Section \ref{sec:Prob1Apriori}, we have defined
\begin{alignat*}{10}
&\bsy{\varepsilon}^{\bq}:=&\, \bsy{\Pi}_{\mbf{V}}\bq - \bq_h, \qquad
\bsy{\varepsilon}^{\bsigma}:=&&\, \bsy{\Pi}_{\mbf{V}}\bsigma - \bsigma_h, \qquad
\varepsilon^u:=&&\, \Pi_W u - u_h \qquad&\text{ \textit{(Projection of the error)}},\\
&\mbf{I}^{\bq}:=&\, \bq-\bsy{\Pi}_{\mbf{V}}\bq, \qquad
\mbf{I}^{\bsigma}:=&&\, \bsigma-\bsy{\Pi}_{\mbf{V}}\bsigma, \qquad
I^u:=&&\, u - \Pi_W u \qquad&\text{ \textit{(Error of the projection)}}.
\end{alignat*}
In addition, using the $L^2$ projection into $M_h$ we have $\varepsilon^{\what{u}}:= P_M u - \what{u}_h$. The vector of error projections $( \bsy{\varepsilon}^{\bq}, \bsy{\varepsilon}^{\bsigma},\varepsilon^u,\varepsilon^{\what{u}})$ belongs to $\mbf{V}_h\times\mbf{V}_h\times  W_h \times M_h$ and satisfies the error equations
	\begin{subequations}\label{error projection-P2}
	\begin{align}
	(\errorsigma, \mbf{v})_{\mc{T}_h} + (\erroru, \nabla \cdot \mbf{v})_{\mc{T}_h} - \langle \errortu, \mbf{v} \cdot \mbf{n} \rangle_{\partial \mc{T}_h} =& - (\mbf{I}^{\bsigma}, \mbf{v} )_{\mc{T}_h} - (\mbf{I}^u,\nabla \cdot \mbf{v})_{\mc{T}_h}  \\
	(\errorq ,\mbf{s})_{\mc{T}_h}  + \left(\kap(\bsigma_h)\, \errorsigma ,\mbf{s} \right)_{\mc{T}_h} =& - (\mbf{I}^{\bq},\mbf{s})_{\mc{T}_h} -\left(\kap(\bsigma)\, \mbf{I}^{\bsigma},\mbf{s}\right)_{\mc{T}_h} \\
	& - \left(\left(\kap(\bsigma) - \kap(\bsigma_h)\right)\Pi_{\mbf{V}}\bsigma,s\right)_{\mc{T}_h} , \\
	-(\errorq, \nabla w)_{\mc{T}_h} + \langle \errortq \cdot \mbf{n},w\rangle_{\partial \mc{T}_h} =&\, (f(u) - f(u_h), w)_{\mc{T}_h}, \\
	\langle\errortu,\mu \rangle_{\Gamma_h} =&\,  \langle \varphi(\bsigma) - \varphi_{h}(\bsigma_h) ,\mu \rangle_{\Gamma_h},  \\
	\langle \errortq \cdot \mbf{n}, \mu\rangle_{\partial \mc{T}_h \setminus \Gamma_h} =&\, 0,
	\end{align}
	\end{subequations}
for all $(\mbf{v},\mbf{s}, w, \mu)\in \mbf{V}_h \times\mbf{V}_h \times W_h \times M_h$. Here, as before,  on $\partial \mc{T}_h$ we have $\errortq \cdot \mbf{n} = \errorq \cdot \mbf{n} +  \tau(\erroru - \errortu)$.

Following the same arguments of Lemma \ref{lem:norm_error} is possible to estimate the magnitude of $(\errorsigma, \errorq, \erroru-\errortu, \varphi-\varphi_{h})$, as measured by the norm $\triple{\cdot}_{\bsigma}$ defined in \eqref{def:triple_norm-P2}. Choosing the vector of approximation errors both as test and trial in the error equations we obtain
    \begin{equation}\label{eq:norm-error_aux_2}
    \begin{array}{l}
    \|\kap^{1/2}(\bsigma_h) \, \errorsigma\|^2_{\compD} + \|\tau^{1/2}(\erroru - \errortu) \|^2_{\partial \mc{T}_h} + \| \kap^{1/2}(\bsigma_h)  \, l^{-1/2} \, (\varphi(\bsigma) - \varphi_{h}(\bsigma_h) )\|^2_{\Gamma_h}  \\[2ex]
    \quad \leq |(\mbf{I}^{\bq}, \errorsigma)_{\mc{T}_h}| + |(\mbf{I}^{\bsigma}, \errorq)_{\mc{T}_h}| + | (\kap(\bsigma) \mbf{I}^{\bsigma}, \errorsigma )_{\mc{T}_h}| + | ((\kap(\bsigma) - \kap(\bsigma_h))\Pi_{\mbf{V}}\bsigma,\errorsigma)_{\mc{T}_h}| +  |\md{T}^{f}| + |\md{T}_{\varphi}|.
    \end{array}
    \end{equation}
The final two terms are defined as $\md{T}^{f}:= (f(u)-f(u_h), \erroru)_{\mc{T}_h}$ and $\md{T}_{\varphi} := \sum_{i=1}^8{|\mr{T}_{\varphi}^i}|$, where

    \begin{equation*}
    \begin{array}{rcl}
    \md{T}_{\varphi}^1 &:=& \langle  \varphi(\bsigma) - \varphi_{h}(\bsigma_h), \delta_{\errorq} \rangle_{\Gamma_h} \\
    \md{T}_{\varphi}^2 &:=& - \langle \varphi(\bsigma) - \varphi_{h}(\bsigma_h), \tau (\erroru - \errortu) \rangle_{\Gamma_h} \\
    \md{T}_{\varphi}^3 &:=& \langle\varphi(\bsigma) - \varphi_{h}(\bsigma_h), \mbf{I}^{\bq} \cdot \mbf{n} \rangle_{\Gamma_h} \\
    \md{T}_{\varphi}^4 &:=& \langle \varphi(\bsigma) - \varphi_{h}(\bsigma_h),\delta_{\mbf{I}^{\bq}} \rangle_{\Gamma_h} \\
    \md{T}_{\varphi}^5 &:=& \langle \varphi(\bsigma) - \varphi_{h}(\bsigma_h), \kap(\bsigma_h)\, (\kap^{-1}(\bsigma) - \kap^{-1}(\bsigma_h))\, \delta_{\Pi_{\mbf{V}}\bq} \rangle_{\Gamma_h}  \\
    \md{T}_{\varphi}^6 &:=& \langle \varphi(\bsigma) - \varphi_{h}(\bsigma_h),  \kap(\bsigma_h)\, (\kap^{-1}(\bsigma) - \kap^{-1}(\bsigma_h))\, \delta_{\mbf{I}^{\bq}}  \rangle_{\Gamma_h}  \\
    \md{T}_{\varphi}^7 &:=& \langle \varphi(\bsigma) - \varphi_{h}(\bsigma_h), \kap(\bsigma_h)\, (\kap^{-1}(\bsigma) - \kap^{-1}(\bsigma_h))\, \mbf{I}^{\bq} \cdot \mbf{n} \rangle_{\Gamma_h}  \\
    \md{T}_{\varphi}^8 &:=& \langle \varphi(\bsigma) - \varphi_{h}(\bsigma_h),  \kap(\bsigma_h)\, (\kap^{-1}(\bsigma) - \kap^{-1}(\bsigma_h))\, \Pi_{\mbf{V}}\bq \cdot \mbf{n} \rangle_{\Gamma_h}.
    \end{array}
    \end{equation*}
By a combined use of arguments similar to those appearing in Lemma \ref{lem:norm_error},  it is possible to deduce 
    \begin{equation}\label{estimate-T_boundary-P2}
    \begin{array}{rl}
    |\md{T}_{\varphi}^1| &\leq \dfrac{1}{2\delta_1} \| \kap^{1/2}(\bsigma_h)\, l^{-1/2}( \varphi(\bsigma) - \varphi_{h}(\bsigma_h)  \|^2_{\Gamma_h} + \dfrac{\delta_1}{6} \, \overline{\kap}^{-1}\, \| \errorq \|^2_{\compD}, \\ 
    |\md{T}_{\varphi}^2| &\leq \dfrac{1}{2\delta_2} \| \kap^{1/2}(\bsigma_h)\, l^{-1/2}( \varphi(\bsigma) - \varphi_{h}(\bsigma_h) \|^2_{\Gamma_h}  + \dfrac{\delta_2}{6}\, \| \tau^{1/2}(\erroru - \errortu)\|^2_{\partial \mc{T}_h},  \\
    |\md{T}_{\varphi}^3| &\leq \dfrac{1}{2\delta_3} \| \kap^{1/2}(\bsigma_h)\, l^{-1/2}( \varphi(\bsigma) - \varphi_{h}(\bsigma_h) \|^2_{\Gamma_h} + \dfrac{\delta_3}{2} \, R \, \underline{\kap}^{-1} \, \| (h^{\perp})^{1/2} \mbf{I}^{\bq} \cdot \mbf{n})\|^2_{\Gamma_h} , \\
    |\md{T}_{\varphi}^4| &\leq \dfrac{1}{2\delta_3} \| \kap^{1/2}(\bsigma_h)\, l^{-1/2}( \varphi(\bsigma) - \varphi_{h}(\bsigma_h) \|^2_{\Gamma_h}  + \dfrac{\delta_3}{6} \, \underline{\kap}^{-1} \, R^2 \,  \| h^{\perp} \partial_n (\mbf{I}^{\bq} \cdot \mbf{n})\|^2_{\compD^c} , \\
    |\md{T}_{\varphi}^5| &\displaystyle \leq \dfrac{1}{2\delta_3} \| \kap^{1/2}(\bsigma_h)\, l^{-1/2}( \varphi(\bsigma) - \varphi_{h}(\bsigma_h) \|^2_{\Gamma_h}  + \dfrac{\delta_3}{3}\, \overline{\kap} \, R^2 \, \sup_{\mbf{x}\in e \subseteq \Gamma_h} \| h^{\perp} \partial_n (\Pi_{\mbf{V}}\bq \cdot \mbf{n} ) \|^2_{l(\mbf{x})} \, L^2 \, \left( \|\errorsigma\|_{\compD}^2 + \| \mbf{I}^{\bsigma} \|_{\compD}^2  \right), \\
    |\md{T}_{\varphi}^6| &\leq \dfrac{1}{2\delta_3} \| \kap^{1/2}(\bsigma_h)\, l^{-1/2}( \varphi(\bsigma) - \varphi_{h}(\bsigma_h) \|^2_{\Gamma_h}  + \dfrac{2\delta_3}{3}\,  \overline{\kap} \, \underline{\kap}^{-2}  \, R^2 \, \| h^{\perp} \partial_n (\mbf{I}^{\bq} \cdot \mbf{n})\|^2_{\compD^c} ,\\
    |\md{T}_{\varphi}^7| &\leq \dfrac{1}{2\delta_3} \| \kap^{1/2}(\bsigma_h)\, l^{-1/2}( \varphi(\bsigma) - \varphi_{h}(\bsigma_h) \|^2_{\Gamma_h}  + 2  \, \delta_3 \, R\, \overline{\kap} \, \underline{\kap}^{-2} \,   \| (h^{\perp})^{1/2} \mbf{I}^{\bq} \cdot \mbf{n})\|^2_{\Gamma_h}  ,\\
    |\md{T}_{\varphi}^8| &\leq \dfrac{1}{2\delta_3} \| \kap^{1/2}(\bsigma_h)\, l^{-1/2}( \varphi(\bsigma) - \varphi_{h}(\bsigma_h) \|^2_{\Gamma_h}  + \delta_3\,  \overline{\kap} \, R\  \|(h^{\perp})^{1/2}\, \Pi_{\mbf{V}} \bq \cdot \mbf{n} \|_{\infty}^2\,  L^2\, \left( \|\errorsigma\|_{\compD}^2 + \| \mbf{I}^{\bsigma} \|_{\compD}^2  \right) ,
    \end{array}
    \end{equation}
and
    \begin{equation}\label{estimate1-aux-error-P2}
    \begin{array}{rl}
	|( \mbf{I}^{\bsigma}, \errorq )_{\mc{T}_h}| &\leq \dfrac{1}{2\, \delta_{4}} \| \errorq \|^2_{\compD} + \dfrac{\delta_{4}}{2}\, \| \mbf{I}^{\bsigma}\|^2_{\compD}, \\
	|( \mbf{I}^{\bq}, \errorsigma )_{\mc{T}_h}| &\leq \dfrac{1}{2\, \delta_{5}} \| \kap^{1/2}(\bsigma_h) \errorsigma \|^2_{\compD} + \dfrac{\delta_5}{2}\, \underline{\kappa}^{-1}\,   \| \mbf{I}^{\bq}\|^2_{\compD}, \\
	|(\kap(\bsigma) \mbf{I}^{\bsigma}, \errorsigma )_{\mc{T}_h}| &\leq \dfrac{1}{2\, \delta_{5}} \| \kap^{1/2}(\bsigma_h) \errorsigma \|^2_{\compD} + \dfrac{\delta_{5}}{2}\, \underline{\kappa}^{-1}\, \overline{\kap}^2   \| \mbf{I}^{\bsigma}\|^2_{\compD}, \\
	|(\kap^(\bsigma) - \kap(\bsigma_h)) \, \Pi_{\mbf{V}}\bsigma,\errorsigma)_{\mc{T}_h}| &\leq \dfrac{1}{2\, \delta_{5}} \,  \| \kap^{1/2}(\bsigma_h) \errorsigma \|^2_{\compD} + \delta_{5}\, \underline{\kap}^{-1}\,  \|\Pi_{\mbf{V}}\bq\|_{\infty}^2 \, L^2\,  \left( \|\errorsigma\|_{\compD}^2 + \| \mbf{I}^{\bsigma} \|_{\compD}^2  \right). \\
	|\md{T}^{f} | &\leq L_{f} \, (\|\erroru\|_{\compD} + \|I^u\|_{\compD})\, \|\erroru\|_{\compD}.
	\end{array}
	\end{equation}
Then, taking the test function $s=\errorq$ in the second equation of \eqref{error projection-P2}, we get
    \begin{equation}\label{estimate2-aux-error-P2}
    \|\errorq\|^2_{\compD} \leq \left( 4\, \overline{\kap} + 8\, \underline{\kap}^{-1}\,  L^2\, \|\Pi_{\mbf{V}}\bq\|^2_{\infty} \right) \, \|\kap^{1/2}(\bsigma_h) \, \errorsigma\|_{\compD}^2 + 4\, \|\mbf{I}^{\bq}\|^2_{\compD} + 4 \, \max\{\overline{\kap}, 2\, L^2\, \|\Pi_{\mbf{V}} \, \bsigma\|^2_{\infty} \} \, \|\mbf{I}^{\bsigma}\|^2_{\compD}.
    \end{equation}
If the Lipschitz constants $L$ and $L_{f}$ are sufficiently small, a direct application of \eqref{estimate2-aux-error-P2} in the first equations of \eqref{estimate-T_boundary-P2} and \eqref{estimate1-aux-error-P2}, together with the choices $\delta_1=1, \delta_2=3, \delta_3=12, \delta_4=24/\overline{\kap}, \delta_5=18$ yield the following estimate for the right hand side of \eqref{eq:norm-error_aux_2}  
    \begin{align}
    \nonumber
    \|\kap^{1/2}(\bsigma_h) \, \errorsigma\|^2_{\compD} + \|\tau^{1/2}(\erroru - \errortu) \|^2_{\partial \mc{T}_h} +  \| \kap^{1/2}(\bsigma_h) & \, l^{-1/2} \, \,(\varphi(\bsigma) - \varphi_{h}(\bsigma_h) )\|^2_{\Gamma_h}   \\[2ex]
    \label{estimate3-aux-error-P2}
    \lesssim\,& \projerrorq^2 + \projerrorsigma^2 +  L_{f} \, (\|\erroru\|_{\compD} + \|I^u\|_{\compD})\, \|\erroru\|_{\compD}.
    \end{align}
In the expression above, the term $\projerrorsigma$ has been defined according to \eqref{def:lambda_q}. By combining \eqref{estimate2-aux-error-P2} and \eqref{estimate3-aux-error-P2} we get
    \begin{equation}\label{eq: est. triple norm P2_a}
    \triple{(\errorsigma, \errorq, \erroru-\errortu, \varphi-\varphi_{h})}^2_{\bsigma_h}     \lesssim \projerrorq^2 + \projerrorsigma^2 +  L_{f} \, (\|\erroru\|_{\compD} + \|I^u\|_{\compD})\, \|\erroru\|_{\compD}.
    \end{equation}
The following result allows us to estimate the term $\erroru$ in \eqref{eq: est. triple norm P2_a}  by means of a duality argument. The proof technique is analogous to the one used for Lemma \ref{lem:estimation E_u}.

\begin{lem}\label{lem:estimation E_u-P2}
Given the regularity condition \eqref{regularity dual problem}, assume that the Lipschitz constant is such that $L_{f}$ is small enough, and consider the discrete spaces to be of polynomial degree $k\geq 1$. Then, 
	\begin{align}\label{estimation E_u-P2}
         \|\erroru\|_{\compD}^2  \lesssim 3\, h  \triple{(\errorsigma,\errorq,\erroru-\errortu, \varphi-\varphi_h)}_{\boldsymbol\sigma_h}^2  + 6 \bigg(  (h + L^2\, h^2) \projerrorq^2 + (L^2+ h)\,  \projerroru^2 \bigg).
	\end{align}
\end{lem}
\begin{proof}
Consider the pair of functions function $\phi$ and $\psi$ satisfying the dual problem  \eqref{eq:DualProblem}. We will use them to define the following terms
    \begin{align*}
     \md{T}_{\bsigma} &:= - ( \bsigma-\bsigma_h, \bsy{\Pi}_{\mbf{V}} \bsy{\phi} - \bsy{\phi})_{\mc{T}_h} + ( (\kap(\bsigma_h) - \kap(\bsigma) )(\mbf{I}^{\bsigma} + \bsy{\Pi}_{\mbf{V}} \bsigma) ,  \nabla \psi)_{\mc{T}_h} , \\
     \md{T}_{\bq} &:= - (\mbf{I}^{\bq} ,\nabla \psi)_{\mc{T}_h}, \\
     \md{T}_{f} &:= (f(u) - f(u_h), \Pi_W \psi)_{\mc{T}_h}, \\
     \md{T}_{u} &:= \langle \errortu, P_M(\bsy{\phi}\cdot \mbf{n}) \rangle_{\Gamma_h} - \langle \errortq \cdot \mbf{n}, \Pi_W \psi\rangle_{\Gamma_h} .
    \end{align*}
With all the definitions above and the equations in \eqref{eq:DualProblem}, it is possible to decompose the inner product between $\varepsilon^u$ and the function $\Theta$ appearing as the source term of the dual problem in the form
	\begin{align}\label{eq:duality-descomp-P2}
	(\varepsilon^u, \Theta)_{\mc{T}_h}& = \md{T}_{\bsigma} + \md{T}_{\bq} +  \md{T}_{f} + \md{T}_{u} ,
	\end{align}
Following arguments similar to the ones applied in Lemma \ref{lem:StabilityS2}, it is possible to bound each of the terms in the decomposition as 
    \begin{align*}
    |\md{T}_{\bsigma}| &\lesssim h^{\min\{1, k\}}\, \|\kap^{1/2}(\bsigma_h) (\errorsigma + \mbf{I}^{\bsigma}) \|_{\compD} \, \|\Theta\|_{\Omega} + L \, ( \|\kap^{1/2}(\bsigma_h) \errorsigma\| + \|\mbf{I}^{\bsigma}\| ) \, \|\Theta\|_{\Omega} \\
    |\md{T}_{\bq}| &\lesssim h^{\min\{1,k\}} \|\mbf{I}\|_{\compD} \, \|\Theta\|_{\Omega} , \\
    |\md{T}_{f}| &\lesssim L_{f} \, \left(\|\erroru\|_{\compD} + \|I^u \|_{\compD} \right)\, \|\Theta\|_{\Omega}
    \end{align*}
The bound for the final term in \eqref{eq:duality-descomp-P2} requires decomposing it in the form $\md{T}_{u} := \sum_{i=1}^{10} \md{T}_{u}^i$, where:
    \begin{equation*}
    \begin{array}{lll}
	\md{T}_{u}^1 := -\langle \kap(\bsigma_h)\,  l^{-1}\, (\varphi(\bsigma) - \varphi_{h}(\bsigma_h)), \psi + l \partial_n \psi \rangle_{\Gamma_h}, &  & \md{T}_{u}^7 := -\langle\tau(\erroru - \errortu), P_M\psi \rangle _{\Gamma_h}, \\
	\md{T}_{u}^2 := -\langle \kap(u_h)\, (\varphi(\bsigma) - \varphi_{h}(\bsigma_h)), (Id_M - P_M)\partial_n \psi \rangle_{\Gamma_h}, &  &  \md{T}_{u}^8 := -\langle \kap(\bsigma_h)\,  (\varphi(\bsigma) - \varphi_{h}(\bsigma_h))\, \delta_{\mbf{I}^{\bq}}, \psi  \rangle_{\Gamma_h}\\
	\md{T}_{u}^3 := \pdual{\delta_{\mbf{I}^q}, \psi}_{\Gamma_h}, &  & \md{T}_{u}^9 := -\langle \kap(\bsigma_h)\,  (\varphi(\bsigma) - \varphi_{h}(\bsigma_h))\, \delta_{\bsy{\Pi}_{\mbf{V}} \bq}, \psi  \rangle_{\Gamma_h} , \\
	\md{T}_{u}^4 := \pdual{\mbf{I}^{\bq} \cdot \mbf{n}, (Id_M-P_M) \psi }_{\Gamma_h} , &  & \md{T}_{u}^{10} := -\langle \kap(\bsigma_h)\,  (\varphi(\bsigma) - \varphi_{h}(\bsigma_h))\, \mbf{I}^{\bq}\cdot \mbf{n}, \psi  \rangle_{\Gamma_h}. \\
	\md{T}_{u}^5 := -\pdual{\tau P_M I^u, \psi}_{\Gamma_h}, &  & \md{T}_{u}^{11} := -\langle \kap(\bsigma_h)\,  (\varphi(\bsigma) - \varphi_{h}(\bsigma_h))\, \bsy{\Pi}_{\mbf{V}} \bq \cdot \mbf{n}, \psi  \rangle_{\Gamma_h}, \\
	\md{T}_{u}^6 := \pdual{\delta_{\bsy{\varepsilon}^{\bq} }, \psi}_{\Gamma_h}. &  & 
	\end{array}
	\end{equation*}
These terms can be estimated by arguments like the ones detailed in Lemma \ref{lem:estimation E_u}.  Finally, taking $\Theta=\erroru$ in \eqref{eq:duality-descomp-P2} and considering the estimates for the components of  $\md{T}_u^i$ it is possible to deduce that 
    \begin{equation*}
        \|\erroru\|_{\compD}^2  \lesssim 3\, h  \triple{(\errorsigma,\errorq,\erroru-\errortu, \varphi-\varphi_h)}_{\boldsymbol\sigma_h}^2  + 6 \bigg(  (h + L^2\, h^2) \projerrorq^2 + (L^2+ h)\,  \projerroru^2 \bigg).
    \end{equation*}
\end{proof}    
The result of the previous Lemma allows us to estimate the error incurred by the HDG approximation by that of the HDG projection onto the discrete space, as we now show. 
\begin{thm}\label{thm:estimate_error-P2}
If $L$ is small enough and the discrete spaces are of polynomial degree $k\geq 1$, then there exists $h_0>0$ such that, for all $h\leq h_0$, we have 
	\begin{eqnarray}\label{eq:estimate_error-P2}
    \triple{ (\errorsigma,\errorq ,\erroru - \errortu, \varphi - \varphi_{h})}_{\boldsymbol\sigma_h}^2  \lesssim \projerrorq^2 + \projerroru^2 + \projerrorsigma^2. 
	\end{eqnarray}
\end{thm}

\begin{proof}
Using simple algebraic arguments, note that the term \eqref{eq: est. triple norm P2_a} can be rewritten as
    \begin{equation*}
        \triple{(\errorsigma, \errorq, \erroru-\errortu, \varphi-\varphi_{h})}^2_{\bsigma}     \lesssim \projerrorq^2 + \projerrorsigma^2 + \frac{3}{2} \, L_{f}\|\erroru\|_{\compD}^2 + \frac{1}{2}\, L_f\, \projerroru \|\erroru\|_{\compD}
    \end{equation*}
Combined the above with the estimate given in the Lemma \ref{lem:estimation E_u-P2}, we can deduce 
    \begin{align*}
     \left( 1 - \frac{9}{2}\,  L_f\, h \, c \right) \triple{(\errorsigma, \errorq, \erroru-\errortu, \varphi-\varphi_{h})}^2_{\bsigma} \lesssim 9\, L_f  \bigg(  (h + L^2\, h^2) \projerrorq^2 + (L^2+ h)\,  \projerroru^2 \bigg) + \projerrorq^2 + \projerrorsigma^2 + \frac{1}{2} \, L_f \projerroru^2,
    \end{align*}
where $c>0$ is a constant independent of $h$ arising from the symbol $\lesssim$.    
Assuming that $L_f$ is small enough and considering that $h\leq h_0$, te proof is concluded.
\end{proof}

As a consequence of this theorem, it follows that the HDG approximation of the linearized systems will indeed achieve optimal order of convergence with respect to the degree of the polynomial basis, provided that the true solutions are smooth enough.

\begin{crl}\label{corol:estimate_error-P2}
Suppose that assumptions of Theorem \ref{thm:estimate_error-P2} hold. If $u\in H^{k+1}(\Omega)$ and $\bq \in  \mbf{H}^{k+1}(\Omega)$, then
\begin{eqnarray*}
\|\boldsymbol\sigma-\boldsymbol\sigma_h\|_{\Omega} + \|\bq-\bq_h\|_{\Omega} + \|u-u_h\|_{\Omega} \leq C h^{k+1} \left( |u|_{k+1,\Omega} + |\bq|_{k+1,\Omega} + |\bsigma|_{k+1,\Omega} \right).
\end{eqnarray*}
\end{crl}

% =========================
\section*{Acknowledgments}
% ==========================
 Nestor S\'anchez is supported by the Scholarship Program of CONICYT-Chile. Manuel E. Solano was partially funded by CONICYT--Chile through FONDECYT project No. 1200569 and by Project AFB170001 of the PIA Program: Concurso Apoyo a Centros Cient\'ificos y Tecnol\'ogicos de Excelencia con Financiamiento Basal. Tonatiuh S\'anchez--Vizuet was partially supported by the National Science Foundation throught the grant NSF-DMS-2137305 ``LEAPS-MPS: Hybridizable discontinuous Galerkin methods for non-linear integro-differential boundary value problems in magnetic plasma confinement''. 
% ================================
\appendix
\setcounter{lem}{0}
\renewcommand{\thelem}{\Alph{section}\arabic{lem}}
% ===============================
%		
% ===============================
\section{HDG projection}\label{sec:HDGprojection}	
% ===============================
%
The HDG projectors introduced by \cite{CoGoSa2010} and their properties have been used extensively throughout the text. Here we provide a quick definition and summary of the properties used in this article.

Consider constants $l_u, l_{\mbf{q}} \in [0,k]$ and functions $(\boldsymbol q,u) \in H^{1+l_q}(T) \times H^{1+l_u}(T)$. Moreover, recall the discrete spaces 
	\begin{align*}
	\mbf{V}_h &:= \{\mbf{v}\in \boldsymbol L^2(\mc{T}_h) : \mbf{v}|_T \in [\md{P}_k(T)]^d, \ \forall \ T \in \mc{T}_h \}, \\
	W_h &:= \{w\in L^2(\mc{T}_h) : w|_T \in \md{P}_k(T), \ \forall \ T \in \mc{T}_h \}, \\
	M_h &:= \{\mu\in L^2(\mc{E}_h) : \mu|_T \in \md{P}_k(e), \ \forall \ e \in \mc{E}_h \},
	\end{align*}
defined in \eqref{eq:PolynomialSpaces} in the text. We will denote by $\bsy{\Pi}(\mbf{q},u):=(\bsy{\Pi}_{\mathrm v}\mbf{q},\Pi_{\mathrm w} u)$ the projection over $\mbf{V}_h\times W_h$ defined by the unique element-wise solutions of
	\begin{subequations}\label{eq:HDGprojector}
	\begin{align}
	(\bsy{\Pi}_{\mathrm v}\mbf{q}, \mbf{v})_T &= (\mbf{q}, \mbf{v})_T &  &\forall \ \mbf{v} \in [\md{P}_{k-1}(T)]^d, \label{properties projector Pi_v} \\
	(\Pi_{\mathrm w} u, w)_T &= (u,w)_T & &\forall \ w\in \md{P}_{k-1}(T), \label{properties projector Pi_w} \\
	\pdual{\bsy{\Pi}_{\mathrm v}\mbf{q}\cdot \mbf{n} + \tau \Pi_{\mathrm w} u, \mu}_{F} &= \pdual{\mbf{q} \cdot \mbf{n} + \tau u, \mu}_F & &\forall \ \mu \in \md{P}_k(F), \label{properties projector Pi_h}
	\end{align}
    \end{subequations}
for every element $T\in \mc{T}_h$, and $F\subset \partial T$. Will will denote the $L^2$ projector into $M_h$ by $P_M$. It was proven in \cite{CoGoSa2010} that when the stabilization function is chosen so that $\tau_T^{\max} := \max \tau|_{\partial T}>0$, then there is a constant $C>0$ independent of $T$ and $\tau$ such that   
    \begin{subequations}\label{eq:projection_error}
    \begin{align}
	\|\bsy{\Pi}_{\mathrm v}\mbf{q} - \mbf{q}\|_T &\leq C h_T^{l_{\mbf{q}}+1} |\mbf{q}|_{\mbf{H}^{l_{\mbf{q}}+1}(T)} +  C h_T^{l_u+1} \tau_T^* |u|_{H^{l_u+1}(T)}, \label{error_projector1}\\
	\|\Pi_{\mathrm w} u - u\|_T &\leq C h_T^{l_u+1} |u|_{H^{l_u+1}(T)} + C \dfrac{h_T^{l_{\mbf{q}}+1}}{\tau_T^{\max}} |\nabla \cdot \mbf{q}|_{H^{l_{\mbf{q}}}(T)}\label{error_projector2}. 
	\end{align}
    \end{subequations}
Here $\tau_T^* := \max \tau|_{\partial T \setminus F^*}$ and $F^*$ is a face of $T$ at which $\tau|_{\partial T}$ is maximum. As is customary, the symbol $|\cdot|_{H^s}$ is to be understood as the Sobolev semi norm of order $s\in\mathbb R$.
%
% ===================================
\section{Auxiliary estimates}
% ===================================
%
% ==========================================
\paragraph{Duality argument}
% ==========================================
We will consider that, given $\Theta \in L^2(\Omega)$, the solution to the auxiliary problem
	\begin{subequations}\label{eq:DualProblem}
	\begin{align}
	&& \kap^{-1}(\zeta) \bsy{\phi}+ \nabla \psi &= 0 & &\text{in } \Omega, \label{eq: dual problem 1} \\
	&& \nabla \cdot \bsy{\phi} &= \Theta & &\text{in } \Omega, \label{eq: dual problem 2} \\
	&& \psi &= 0 & &\text{on } \partial \Omega ,\label{eq: dual problem 3}
	\end{align}
	\end{subequations}
satisfies the regularity estimate
	\begin{equation}\label{regularity dual problem}
	\| \bsy{\phi} \|_{H^1(\Omega)} + \| \psi \|_{H^2(\Omega)} \leq C_{reg} \|\Theta\|_{\Omega}.
	\end{equation}	
where $ C_{reg} >0$ depends on the domain $\Omega$. Moreover, if $\Omega_h$ is a subdomain of $\Omega$ with boundary $\Gamma_h:=\partial\Omega_h$, $l(x)$ is the length of the transfer path connecting $\Gamma:=\partial\Omega$ to $\Gamma_h$ as defined in Section \ref{sec:ExtendedDomain}, $P_M$ is the $L^2-$projector onto the discrete space $M_h$ defined in \eqref{eq:PolynomialSpaces}, and $Id_M$ is the identity operator in $M_h$ we have

\begin{lem}\label{Lema auxiliar para T_u}\cite[Lemma 5.5]{CoQiuSo2014}
Suppose  Assumption \eqref{eq:S4} and the elliptic regularity inequality \eqref{regularity dual problem} hold. Then, there exists a constant $\tilde{c}>0$, such that:
	\begin{subequations}
	\begin{align}
	\|{(h^{\perp})}^{-1/2} (Id_M-P_M) \psi \|_{\Gamma_h} &\leq \tilde{c} \,  h \|\Theta\|_{\Omega}, \label{estim-aux for Tu-p1} \\
	\|l^{1/2} (Id_M-P_M)\partial_n \psi \|_{\Gamma_h} &\leq \tilde{c} \,  R^{1/2} \, h \|\Theta\|_{\Omega}, \label{estim-aux for Tu-p2} \\
	\|l^{-3/2} (\psi + l \partial_n \psi) \|_{\Gamma_h} &\leq \tilde{c} \,  \|\Theta\|_{\Omega}, \label{estim-aux for Tu-p3} \\
	\| l^{-1} \psi \|_{\Gamma_h} &\leq \tilde{c} \,  \|\Theta\|_{\Omega}. \label{estim-aux for Tu-p4}
	\end{align}
	\end{subequations}
\end{lem}	
%
% ==============================
\paragraph{Function Delta}
% ==============================
%
 For any smooth enough function $\mbf{v}$ defined in $T^e\cup T_{ext}^e$ and $\mbf{x}\in\Gamma_h$ we set
	\begin{equation}\label{def:delta}
	\delta_{\mbf{v}} (\mbf{x}) := \dfrac{1}{l(\mbf{x})} \int_0^{l(\mbf{x})} [\mbf{v}(\mbf{x} + \mbf{n}s) - \mbf{v}(\mbf{x}) ] \cdot \mbf{n} ds.
	\end{equation}
which hold for each $e\in \mc{E}_h^{\partial}$ (cf. \cite[Lemma 5.2]{CoQiuSo2014}):
	\begin{subequations}\label{ineq:deltav}
	\begin{align}
	\label{ineq: est delta 1}
	\| l^{1/2} \, \delta_{\mbf{v}} \|_e &\leq \dfrac{1}{\sqrt{3}} \, r_e^{3/2} \, C_{ext}^e \, C_{inv}^e \, \|\mbf{v}\|_{T^e}  & &\forall \ \mbf{v} \in [\md{P}_k(T)]^d, \\
	\label{ineq: est delta 2}
	\| l^{1/2}\,\delta_{\mbf{v}} \|_e &\leq \dfrac{1}{\sqrt{3}} \, r_e \, \| h^{\perp} \partial_n \mbf{v} \cdot \mbf{n}\|_{T_{ext}^e} & &\forall \ \mbf{v} \in [H^1(T)]^d. \\
	\label{ineq: est delta 3}
	\|l^{1/2}\, \delta_{\mbf{v}}\|_{\infty} &\leq \dfrac{1}{\sqrt{3}} \, r_e \,  \sup_{\mbf{x}\in e} \|h_e^{\perp}\,\partial_n \mbf{v}\cdot \mbf{n} \|_{l(\mbf{x})} &&\forall \ \mbf{v} \in [H^1(T)]^d.
	\end{align}
	\end{subequations}

% ======================
\bibliography{biblio}
\bibliographystyle{abbrv}
\include{biblio}

% ========================= 

\end{document}